\documentclass{siamltex}
\usepackage{amsmath}
\usepackage{amsfonts}
\usepackage{amssymb}
\usepackage{graphicx}
\usepackage{color}
\usepackage{bm}
\usepackage{verbatim}
\usepackage{hyperref}
\usepackage{bmpsize}

\usepackage{bigints}

\usepackage{graphics,color}

\usepackage{morefloats}
\usepackage{caption} 
\allowdisplaybreaks

\newenvironment{MSC2010}{\begin{@abssec}{MSC (2010)}}{\end{@abssec}}
\newcommand{\be}{\begin{equation}}
\newcommand{\ee}{\end{equation}}
\newcommand{\bea}{\begin{eqnarray}}
\newcommand{\eea}{\end{eqnarray}}
\newcommand{\beas}{\begin{eqnarray*}}
\newcommand{\eeas}{\end{eqnarray*}}

\newcommand{\vertiii}[1]{{\left\vert\kern-0.25ex\left\vert\kern-0.25ex\left\vert #1 
    \right\vert\kern-0.25ex\right\vert\kern-0.25ex\right\vert}}

\begin{document}
\title{A second-order time-stepping scheme for simulating ensembles of parameterized flow problems}

\author{
Max Gunzburger\thanks{Department of Scientific Computing,  
Florida State University,
Tallahassee, FL 32306-4120, {\tt mgunzburger@fsu.edu}. Research supported by the U.S. Department of Energy grant DE-SC0009324, U.S. Air Force Office of Scientific Research grant FA9550-15-1-0001, and by a Defense Advanced Projects Agency contract administered under the Oak Ridge National Laboratory subcontract 4000145366. }
\and Nan Jiang\thanks{Department of Mathematics and Statistics,  
Missouri University of Science and Technology,
Rolla, MO 65409-0020, {\tt jiangn@mst.edu}. }
\and Zhu Wang \thanks{Department of Mathematics, 
University of South Carolina, Columbia, SC 29208, {\tt wangzhu@math.sc.edu}. 
Research supported by the U.S. Department of Energy grant DE-SC0016540, and the U.S. National Science Foundation grant DMS-1522672.
}
}
%\date{\today}
\maketitle

\begin{abstract}
We consider settings for which one needs to perform multiple flow simulations based on the Navier-Stokes equations, each having different values for the physical parameters and/or different initial condition data, boundary conditions data, and/or forcing functions. For such settings, we propose a second-order time accurate ensemble-based method that to simulate the whole set of solutions, requires, at each time step, the solution of only a single linear system with multiple right-hand-side vectors. Rigorous analyses are given proving the conditional stability and error estimates for the proposed algorithm. 
Numerical experiments are provided that illustrate the analyses. 
\end{abstract}


\begin{MSC2010}
65M60, 76D05
\end{MSC2010}
%======================================
\section{Introduction}
%======================================

Many computational fluid dynamics applications require multiple simulations of a flow under different input conditions. For example, the ensemble Kalman filter approach used in data assimilation first simulates a forward model a large number of times by perturbing either the initial condition data, boundary condition data, or uncertain parameters, then corrects the model based on the model forecasts and observational data. A second example is the construction of low-dimensional surrogates for solutions of partial differential equation (PDE) such as sparse-grid interpolants or proper orthogonal decomposition approximations for which one has to first obtain expensive approximations of solutions corresponding to several parameter samples.  Another example is sensitivity analyses of solutions for which one often has to determine approximate solutions for a number of perturbed inputs such as the values of certain physical parameters. In this paper, we consider such applications and develop a second-order time-stepping scheme for efficiently simulating an ensemble of flows. In particular, we consider the setting in which one wishes to determine the PDE solutions for several different values of the physical parameters and with several different choices of initial condition and boundary condition data and forcing functions appearing in the PDE model. 
 
The ensemble algorithm we use was first developed in \cite{JL14} to find a set of $J$ solutions of the Navier-Stokes equations (NSE) subject to different initial condition and forcing functions. The main idea is that, based on the introduction of an ensemble average and a special semi-implicit time discretization, the discrete systems for the multiple flow simulations share a common coefficient matrix. Thus, instead of solving $J$ linear system with $J$ right-hand sides (RHS), one only need solve one linear system with $J$ RHS. This leads to great computational saving in linear solvers when either the $LU$ factorization (for small-scale systems) or a block iterative algorithm (for large-scale systems) is used. High-order ensemble algorithms were designed in \cite{J15,J16}. For high Reynolds number flows, ensemble regularization methods and a turbulence model based on ensemble averaging have been developed in \cite{J15,JL15,TNW16,JKL15}. The method has also been extended to simulate MHD flows in \cite{MR16} and to develop ensemble-based reduced-order modeling techniques in \cite{GJS16a,GJS16b}.  In \cite{GJW17}, the authors proposed a first-order ensemble algorithm that deals with a number of flow simulations subject to not only different initial condition, boundary conditions, and/or body force data, but also distinct viscosity coefficients appearing in the NSE model. In this paper, we follow the same direction and develop an ensemble scheme with higher accuracy. 

To begin, consider an ensemble of incompressible flow simulations on a bounded domain subject to Dirichlet boundary conditions. The $j$-th member of the ensemble is a simulation associated with the positive viscosity coefficient $\nu_j$, initial condition data $u_{j}^{0}$, boundary condition data $g_j$, and body force $f_{j}$. Any an all of this data may vary from one simulation to another. Then, for $j=1,...,J$, we need to solve 
\begin{equation}
\label{eq:NSE}
\begin{array}{rcll}
u_{j,t}+u_{j}\cdot\nabla u_{j}-\nu_j\triangle u_{j}+\nabla p_{j}  &=& f_{j}(x,t) \quad &\text{ in }\Omega\times [0, \infty) \text{,}\\
\nabla\cdot u_{j}  &=& 0 \quad &\text{ in }\Omega \times [0, \infty) \text{,}\\
u_{j}  &=& g_j(x, t) \quad &\text{ on }\partial\Omega\text{,}\\
u_{j}(x,0)  &=& u_{j}^{0}(x) \quad & \text{ in }\Omega\text{.}
\end{array}
\end{equation}
Because of the nonlinear convection term in the model, either implicit or semi-implicit schemes are always preferred for time discretizations to avoid stability issues. At best, at each time step, for each $j$ a different linear system has to be solved; solving $J$ linear systems per time step requires a huge computational effort. Hence, we propose a new, second-order accurate in time, and more efficient numerical scheme that improves the efficiency by employing a single coefficient matrix for all the ensemble members.

To keep the exposition simple, we consider a uniform time step $\Delta t$ and let $t_n=n\Delta t$ for $n=0,1,\ldots$. We then consider the semi-discrete in time ensemble of systems 
\begin{equation}
\label{Second-Order}
\begin{aligned}
\frac{3u_{j}^{n+1}-4u_{j}^{n}+u_j^{n-1}}{2\Delta t}+\overline{u}^{n}\cdot\nabla &u_{j}^{n+1}
+u_j^{\prime\, n}\cdot\nabla (2u_{j}^{n}-u_j^{n-1}) +\nabla p_{j}^{n+1}\\
&-\overline{\nu}\Delta u_{j}^{n+1}-\left(\nu_j-\overline{\nu}\right)\Delta (2u_{j}^{n}-u_j^{n-1})=f_{j}^{n+1}\text{, }\\
\nabla\cdot u_{j}^{n+1}=0,\hspace{1.5cm}&
\end{aligned}
\end{equation}
where $u_j^{n}$, $p_j^n$ and $f_j^n$ denote approximations of $u_j(\cdot,t_n)$, $p_j(\cdot,t_n)$ and $f_j(\cdot,t_n)$ of \eqref{eq:NSE}, respectively, at time $t_n$. In \eqref{Second-Order}, $\overline{u}^{n}$ and $\overline{\nu}$ denote the ensemble mean of the velocity field and viscosity coefficient, respectively, defined by
\[
\overline{u}^{n}:=\frac{1}{J}\sum_{j=1}^{J} \left(2u_{j}^{n}-u_j^{n-1}\right) \qquad \text{and}\qquad \overline{\nu}:=\frac{1}{J}\sum_{j=1}^{J}\nu_{j}
\]
and $u_j^{\prime\, n}$ represents the fluctuation defined by 
\[
u_j^{\prime\, n}=2u_j^n-u_j^{n-1}-\overline{u}^n.
\]
It is easy to see that the coefficient matrix in a spatial discretization of \eqref{Second-Order} does not depend on $j$. Thus, all the members in the ensemble share a common coefficient matrix. To advance one time step, one only need solve a single linear system with $J$ RHS vectors, which is more efficient than solving $J$ individual simulations. 

In what follows, we present a rigorous theoretical analysis of the stability and second-order accuracy of the scheme. In Section \ref{sec:not}, we provide some notations and preliminaries; in Section \ref{sec:sta}, the stability conditions of the scheme are obtained; and in Section \ref{sec:err}, an error estimate is derived. Then, several numerical experiments are presented in Section \ref{sec:num}.  
%Furthermore, when the size of ensemble is huge, it can be subdivided into $p$ sub-ensembles, balancing memory, communication and computations, then \eqref{Second-Order} can be applied to each.
%Further, the choice of the ensemble data $u_{j}^{0}$ and $f_{j}$ is application dependent.

Note that the inhomogeneity of boundary conditions will not pose any difficulty to the numerical analysis, since we can homogenize it and, thus, consider the NSE solutions with the homogeneous Dirichlet boundary conditions and a new body force (see Section 5.4 in \cite{BS08} for Poisson's equation, while a time-dependent problem can be treated analogously). Therefore, to simplify the presentation, we assume flow boundary conditions to be homogeneous ($g_j= 0$) in the following derivation and analysis of the proposed ensemble algorithm. But the argument can be naturally extended to the inhomogeneous cases. Furthermore, the flow boundary conditons are inhomogeneous in our first numerical experiment presented in Section \ref{sec:num}.

%===========================================
\section{Notation and preliminaries\label{sec:not}}
%===========================================

Let $\Omega$ be an open, regular domain in $\mathbb{R}^{d}$ $(d=2 \text{ or }%
3)$. 
The space $L^{2}(\Omega)$ is equipped with the norm $\|\cdot\|$ and inner product
$(\cdot, \cdot)$. 
Denote by $\|\cdot\|_{L^{p}}$ and $\|\cdot\|_{W_{p}^{k}}$, respectively, the norms for $L^{p}(\Omega)$ and the Sobolev space $W^{k}_{p}(\Omega)$. 
Let $H^{k}(\Omega)$ be the Sobolev space $W_{2}^{k}(\Omega)$ equipped with the 
norm $\|\cdot\|_{k}$. For functions $v(x,t)$ defined on $(0,T)$, we define 
$(1\leq m<\infty)$
\[
\| v \|_{\infty,k} \text{ }:=EssSup_{[0,T]}\| v(t,\cdot)\|_{k} \text{ \,\,\,\, and
  \,\,\,\,}\|v\|_{m,k} \text{ }:= \left(  \int_{0}^{T}\|v(t,\cdot)\|_{k}^{m}\, dt\right)
^{1/m} \text{ .}%
\]
Given a time step $\Delta t$, let $v^n= v(t_n)$ and define the  discrete norms 
\[
\vertiii{v}_{\infty,k}=\max\limits_{0\leq n\leq N}\Vert v^{n}\Vert_{k} \text{ \,\,\,\,  and 
 \,\,\,\, }
\vertiii{v}_{m,k}:= \left(\sum_{n=0}^{N}\|v^{n}\|_{k}^{m}\Delta t\right)^{1/m}.
\]
%Œ
Denote by $H^{-k}(\Omega)$ the dual space of bounded linear functions on
$H^{k}_{0}(\Omega)$. A norm for $H^{-1}(\Omega)$  is given by
\[
\|f\|_{-1}=\sup_{0\neq v\in H^{1}_{0}(\Omega)}\frac{(f,v)}{\Vert\nabla v\Vert}
\text{ .}
\]
We choose the velocity space $X$ and pressure space $Q$ to be 
\[
X:=(H_{0}^{1}(\Omega))^{d} \text{ \,\,\,\,  and 
 \,\,\,\, } Q:=L_{0}^{2}(\Omega).
\]
The space of weakly divergence free functions is then
\[
V\text{ }:=\{v\in X:\,\, (\nabla\cdot v,q)=0\text{ , }\forall\, q\in Q\}.
\]

A weak formulation of (\ref{eq:NSE}) reads: find $u_j:\,[0,T]\rightarrow X$ and 
$p_j:\, [0,T]\rightarrow Q$ for a.e. $t\in(0,T]$ satisfying, for $j=1, ..., J$,
\begin{equation}
\label{eq:weak}
\begin{aligned}
(u_{j,t},v)+(u_{j}\cdot\nabla u_{j},v)+\nu_j(\nabla u_{j},\nabla v)-(p_{j}%
,\nabla\cdot v)  &= (f_{j},v), &\forall\, v\in X, \\
(\nabla\cdot u_{j},q)  &= 0,   &\forall\, q\in Q
\end{aligned}
\end{equation}
with $u_{j}(x,0)=u_{j}^{0}(x)$.

For the spatial discretization, we use a finite element (FE) method. 
However, the results can be extended to many other variational methods without much difficulty. 
Denote by $X_{h}\subset X$ and $Q_{h}\subset Q$ the conforming velocity and pressure FE spaces on an edge to edge triangulation of $\Omega$ with $h$ denoting the maximum diameter of triangles. 
Assume that the pair of spaces $(X_h,Q_h)$ satisfy the discrete inf-sup (or $LBB_h$) condition that is required to guarantee the stability of FE approximations. 
We also assume that the FE spaces satisfy the following approximation properties \cite{Layton08}: 
\begin{align}
\inf_{v_h\in X_h}\| v- v_h \|&\leq C h^{k+1}\Vert u \Vert_{k+1}	   &\forall\, v\in [H^{k+1}(\Omega)]^d\,, \label{Interp1}\\
\inf_{v_h\in X_h}\| \nabla ( v- v_h )\|&\leq C h^k \Vert v\Vert_{k+1}&\forall\, v\in [H^{k+1}(\Omega)]^d\,, \label{interp2}\\
\inf_{q_h\in Q_h}\| q- q_h \|&\leq C h^{s+1}\Vert p\Vert_{s+1}	   &\forall\, q\in H^{s+1}(\Omega)\,,      \label{interp3}
\end{align}
where the generic constant $C>0$ is independent of mesh size $h$. 
One example for which the $LBB_h$ stability condition is satisfied is the family of Taylor-Hood $P^{s+1}$-$P^{s}$ element pairs, for $s\geq 1$ \cite{Max89}. 
The discrete divergence free subspace of $X_{h}$ is
\[
V_{h}\text{ }:=\{v_{h}\in X_{h}:\,(\nabla\cdot v_{h},q_{h})=0\text{ , }\forall\,
q_{h}\in Q_{h}\}.
\]
We assume the mesh and FE spaces satisfy the following standard inverse inequality (typical for locally quasi-uniform meshes and standard FEM spaces, see, e.g., \cite{BS08}): for all $v_{h}\in X_{h}$,%
\begin{align}
h\Vert\nabla v_{h}\Vert &  \leq C_{(inv)}\Vert v_{h}\Vert. \label{inverse}
\end{align}
Define the  explicitly skew symmetric trilinear form
\[
b^{\ast}(u,v,w):=\frac{1}{2}(u\cdot\nabla v,w)-\frac{1}{2}(u\cdot\nabla w,v), 
\]
which satisfies the bounds (\cite{Layton08})
\begin{gather}
b^{\ast}(u,v,w)\leq C(\Omega) \left(\Vert \nabla u\Vert\Vert u\Vert\right)^{1/2}\Vert\nabla v\Vert\Vert\nabla
w \Vert\text{, }\quad \forall\, u, v, w \in X ,\label{In1}\\
b^{\ast}(u,v,w)\leq C(\Omega) \Vert \nabla u\Vert\Vert\nabla v\Vert\left(\Vert\nabla
w \Vert\Vert w\Vert\right)^{1/2},\quad \forall\, u, v, w \in X, \label{In2}
\end{gather}
where $C$ is a constant depending on the domain.
%Let $t^{n}=n\Delta t$, $n=0,1, \ldots ,N$ and $T:=N\Delta t$. 
Denote the exact solution and FE approximate solution at $t=t_n$ to be $u_{j}^{n}$ and $u_{j, h}^{n}$, respectively.

The fully discrete finite element discretization of (\ref{Second-Order}) at $t_{n+1}$ is as follows: given
$u_{j,h}^{n}$,  find $u_{j,h}^{n+1}\in X_{h}$ and $p_{j,h}^{n+1}\in Q_{h}$ satisfying
\begin{equation}
\label{Second-Order-h}
\begin{aligned}
&\hspace{-.5cm}\Big(\frac{3u_{j,h}^{n+1}-4u_{j,h}^{n}+u_{j,h}^{n-1}}{2\Delta t},v_{h}\Big)+b^{\ast}(\overline{u}_{h}^{n},u_{j,h}^{n+1},v_{h})\\
&+b^{\ast}(2u_{j,h}^{n}-u_{j,h}^{n-1}-\overline{u}_{h}^{n},2u_{j,h}^{n}-u_{j,h}^{n-1}, v_{h})
-(p_{j,h}^{n+1},\nabla\cdot v_{h})\\
&+\overline{\nu}(\nabla u_{j,h}^{n+1},\nabla
v_{h})+\left(\nu_j-\overline{\nu}\right)(\nabla (2u_{j,h}^{n}-u_{j,h}^{n-1}),\nabla
v_{h})=(f_{j}^{n+1},v_{h})\text{, } \quad\forall\, v_{h}\in X_{h},\\
&\big(\nabla\cdot u_{j,h}^{n+1},q_{h}\big)=0, \hspace{7.8cm} \forall\, q_{h}\in Q_{h}.
\end{aligned}
\end{equation}

This is a two-step method, which needs $u_{j, h}^0$ and $u_{j, h}^1$ to start the time integration; 
$u_{j, h}^0$ is determined by the initial condition, and $u_{j, h}^1$ can be computed by the first-order ensemble algorithm developed by the authors in \cite{GJW17} or by using the usual, non-ensemble time stepping methods to compute each individual simulation at the very first time step. 

\section{Stability Analysis\label{sec:sta}}

We begin by proving the conditional, nonlinear, long time stability  
of (\ref{Second-Order-h}) under conditions on the time step and parameter deviation: 
for any $j= 1, \ldots, J$, there exists $0\leq\mu<1$ and $0< \epsilon\leq 2-2\sqrt{\mu}$ such that 
\begin{align}
C\frac{\Delta t}{\overline{\nu} h}\left\Vert\nabla u_{j,h}^{\prime\, n}\right\Vert^{2}
&\leq \frac{(2-2\sqrt{\mu}-\epsilon)\sqrt{\mu}}{2(\sqrt{\mu}+\epsilon)} 
\quad \text{and}
 \label{ineq:CFL-h1}
 \\
 \frac{|\nu_j -\overline{\nu}| }{\overline{\nu}} &\leq \frac{\sqrt{\mu}}{3},
 \label{ineq:CFL-h2}
\end{align}
where $C$ denotes a generic constant depending on the domain and the minimum angle of the mesh.

\begin{theorem}[Stability] 
\label{th:stability}
The ensemble scheme (\ref{Second-Order-h}) is stable provided the conditions (\ref{ineq:CFL-h1})-(\ref{ineq:CFL-h2}) hold. 
In particular, for $j= 1, \ldots, J$ and for any $N\geq 2$, we have 
%$0\leq  \beta < \frac{1}{2}$ and
%
\begin{equation}
\label{stability_result}
\begin{split}
&\frac{1}{4}\left(\Vert u_{j,h}^{N}\Vert^{2}+\Vert 2u_{j,h}^{N}-u_{j,h}^{N-1}\Vert^2\right)+\frac{1}{8}\sum_{n=1}^{N-1}\Vert u_{j,h}^{n+1}-2u_{j,h}^{n}+u_{j,h}^{n-1}\Vert^{2}\\
&\,\,+\overline{\nu}\Delta t  \frac{\sqrt{\mu}+\epsilon}{2-\sqrt{\mu}}\left(\frac{\sqrt{\mu}}{2}\frac{2+\epsilon}{\sqrt{\mu}+\epsilon}-\frac{3\vert \nu_j-\overline{\nu}\vert}{2 \overline{\nu}}\right)
\|\nabla u_{j,h}^{N}\|^{2}\\
&\leq\sum_{n=1}^{N-1}\frac{\sqrt{\mu}+\epsilon}{2\epsilon(2-\sqrt{\mu})}\frac{\Delta t}{\overline{\nu}}\|f_{j}^{n+1}\|_{-1}^{2}+ \frac{1}{4}\left(\Vert u_{j,h}^{1}\Vert^{2}+\Vert 2u_{j,h}^{1}-u_{j,h}^{0}\Vert^2\right)\\
&\,\,+\overline{\nu}\Delta t  \frac{\sqrt{\mu}+\epsilon}{2-\sqrt{\mu}}\left(\frac{\sqrt{\mu}}{2}\frac{2+\epsilon}{\sqrt{\mu}+\epsilon}-\frac{3\vert \nu_j-\overline{\nu}\vert}{2 \overline{\nu}}\right)\|\nabla u_{j,h}^{1}\|^{2}.
\end{split}
\end{equation}
%\begin{align}
%&\frac{1}{2}\|u_{j,h}^{N}\|^{2}+\frac{1}{4}\sum_{n=0}^{N-1}\|u_{j,h}%
%^{n+1}-u_{j,h}^{n}\|^{2}+\overline{\nu}\Delta
%t (\frac{(3-\sqrt{\mu})\sqrt{\mu}}{2}-\frac{\vert \nu_j-\overline{\nu}\vert}{2 \overline{\nu}})\|\nabla u_{j,h}^{N}\|^{2}\label{Stability}\\
%&\leq\sum_{n=0}^{N-1}\frac{\Delta t}{\nu}\|f_{j}^{n+1}\|_{*}^{2}+ \frac{1}%
%{2}\|u_{j,h}^{0}\|^{2}+\overline{\nu}\Delta
%t (\frac{(3-\sqrt{\mu})\sqrt{\mu}}{2}-\frac{\vert \nu_j-\overline{\nu}\vert}{2 \overline{\nu}})\|\nabla u_{j,h}^{0}\|^{2}, \text{  j= 1,...,J
%.}\nonumber
%\end{align}
\end{theorem}

\begin{proof}
See Appendix \ref{proofstab}. 
\end{proof}

\section  {Error Analysis\label{sec:err}}

In this section we derive the numerical error estimate of the proposed ensemble scheme \eqref{Second-Order-h}. We first give a lemma on the estimate of the consistency error of the backward differentiation formula, which will be used in the error analysis for the fully discrete ensemble scheme. 

\begin{lemma}\label{lm} For any $u\in H^3(0,T;L^2(\Omega))$, we have that
\begin{gather}
\Big\Vert \frac{3u^{n+1}-4u^n+u^{n-1}}{2\Delta t}-u_t^{n+1}\Big\Vert^2\leq \frac{5}{2}\Delta t^3 \left(\int_{t^{n-2}}^{t^{n+1}}\Vert u_{ttt}\Vert^2 dt\right).\label{cons1}
\end{gather}
\end{lemma}

\begin{proof} 
The proof is given in Appendix \ref{prooflem1}.
\end{proof}

Assuming that $X_{h}$ and $Q_{h}$ satisfy the $LBB_{h}$ condition, then the ensemble scheme 
\eqref{Second-Order-h} is equivalent to: \textit{for} $n=1, ..., N-1$, \textit{find}
$u_{j,h}^{n+1} \in V_{h}$ \textit{such that}
\begin{equation}
\label{error-V}
\begin{aligned}
&\hspace{-.5cm}\left(\frac{3u_{j,h}^{n+1}-4u_{j,h}^{n}+u_{j,h}^{n-1}}{2\Delta t},v_{h}\right)+b^{\ast}(\overline{u}_{h}^{n},u_{j,h}^{n+1},v_{h})+\overline{\nu}(\nabla u_{j,h}^{n+1},\nabla
v_{h})\\
&+b^{\ast}(2u_{j,h}^{n}-u_{j,h}^{n-1}-\overline{u}_{h}^{n},2u_{j,h}^{n}-u_{j,h}^{n-1}, v_{h})\\
&+\left(\nu_j-\overline{\nu}\right)(\nabla (2u_{j,h}^{n}-u_{j,h}^{n-1}),\nabla
v_{h})=(f_{j}^{n+1},v_{h})\text{, } \qquad\forall v_{h}\in V_{h}.
\end{aligned}
\end{equation}

\noindent To analyze the rate of convergence of the approximation, we assume the
following regularity assumptions on the NSE
\begin{gather*}
u_{j} \in L^{\infty}\left(0,T;H^{1}(\Omega)\right)\cap H^{1}\left(0,T;H^{k+1}(\Omega)\right)\cap
H^{2}\left(0,T;H^{1}(\Omega)\right),\\
p_{j} \in L^{2}\left(0,T;H^{s+1}(\Omega)\right), \text{and }f_j \in L^{2}%
\left(0,T;L^{2}(\Omega)\right).
\end{gather*}
Let $e_{j}^{n}=u_{j}^{n}-u_{j,h}^{n}$ be the error between the true solution
and the approximate solution. We then have the following error estimates.

\begin{theorem}[Error Estimate]\label{th:err} 
For any $j= 1, \ldots, J$, under the stability conditions of \eqref{ineq:CFL-h1}-\eqref{ineq:CFL-h2} for some $\mu$ and $\epsilon$ satisfying $0\leq\mu<1$ and $0< \epsilon\leq 2-2\sqrt{\mu}$, 
% if for some $\mu$, $0\leq\mu<1$, and some $\epsilon$, $0< \epsilon\leq 2-2\sqrt{\mu}$, \eqref{ineq:CFL-h1}-\eqref{ineq:CFL-h2} hold
%{\color{blue} NO NEED TO REPEAT THESE CONDITIONS, IS THERE?}
%{\color{red} the following time-step condition and parameter deviation condition hold
%\begin{align}
%C\frac{\Delta t}{\overline{\nu} h}\left\Vert\nabla u_{j,h}^{\prime n}\right\Vert^{2}
%&\leq \frac{(2-2\sqrt{\mu}-\epsilon)\sqrt{\mu}}{2(\sqrt{\mu}+\epsilon)} ,\label{conv1}
% \\
% \frac{|\nu_j -\overline{\nu}| }{\overline{\nu}} &\leq \frac{\sqrt{\mu}}{3},\label{conv2}
%\end{align}
%} 
there exists a positive constant $C$ independent of the time step $\Delta t$
% \textcolor{red}{Nan: it also doesn't depend on $h$, right? How about $C_0$?} \textcolor{blue}{right, the $C$ in (4.5) does not depend on $h$; But the $C$ in (4.3) does, it depends on the minimum angle of the mesh, because of the use of the inverse inequality; $C_0=\frac{1}{17} \frac{\epsilon} {\sqrt{\mu}+\epsilon}( 1-\frac{\sqrt{\mu}}{2})$, see the line below (B.21), page 18, I've made changes below in (4.5)}
such that 
%\begingroup
\allowdisplaybreaks
\begin{align}
&\frac{1}{4}\|e_j^{N}\|^{2}+ 2C_0\overline{\nu}\Delta t \|\nabla e_{j}^{N}\|^{2}\nonumber\\
&\leq e^{\frac{CT}{\nu^{3}}}
\bigg\{\frac{1}{4}\left(\|e_{j}^{1}\|^{2}+\Vert 2e_{j}^{1}-e_{j}^{0}\Vert^2\right)+2C_0\overline{\nu}\Delta t \|\nabla e_{j}^{1}\|^{2}\nonumber \\
&\,\,+C_0\overline{\nu} \Delta t \Vert\nabla e_{j}^{0}\Vert
^{2}+ C\overline{\nu}^{-1}h^{2k}\vertiii{ u_j}^4_{4, k+1}+C\overline{\nu}^{-1}h^{2k} \label{ineq:err00}\\
&\,\,+C\Delta t^4\frac{\vert \nu_j -\overline{\nu}\vert^2}{\overline{\nu}}\vertiii{ \nabla u_{j,tt}}^2_{2,0}+C\overline{\nu}^{-1}\Delta t^{4}\vertiii{u_{j,tt}}_{2,0}^{2}+ C\overline{\nu}^{-1}h^{2k}\vertiii{ u_j}^2_{2,k+1}\nonumber\\
&\,\,+C h \Delta t^3 \vertiii{ \nabla u_{j,tt}}^2_{2,0}+C h^{2k+1} \Delta t^3 \vertiii{ \nabla u_{j,tt}}^2_{2,k+1}+C\overline{\nu}^{-1} h^{2s+2}\vertiii{p_{j}}_{2,s+1}^{2}\nonumber\\
&\,\,+C\overline{\nu}^{-1} h^{2k+2}\vertiii{u_{j,t}}_{2,k+1}^{2}+C\overline{\nu}h^{2k}\vertiii{ u_j }^2_{2,k+1}+C\frac{\vert \nu_j - \overline{\nu}\vert^2}{\overline{\nu}}h^{2k}\vertiii{ u_j}^2_{2,k+1}\nonumber\\
&\,\,+C\overline{\nu}^{-1}\Delta t^4\vertiii{ \nabla u_{j,ttt}}^2_{2,0} 
\bigg\}+Ch^{2k+2}\vertiii{ u_{j}}^2_{\infty,k+1}+C\overline{\nu}h^{2k}\Delta t\vertiii{ u_{j}}^2_{\infty,k+1}\text{ ,} \nonumber
\end{align}
where $C_0=\frac{1}{17} \frac{\epsilon} {\sqrt{\mu}+\epsilon}( 1-\frac{\sqrt{\mu}}{2})$.
%\endgroup
\end{theorem}

\begin{proof}
See Appendix \ref{proofthm1}.
\end{proof}

It is seen that when the popular $P^2$-$P^1$ Taylor-Hood FE is used for $X_h$ and $Q_h$, that is, $k=2$ and $s=1$, we have the following optimal convergence results: 
\begin{corollary}
Suppose the $P^2$-$P^1$ Taylor-Hood FE pair is used for the spatial discretization and assume that the initial errors $\|u_j^{0}-u_{j, h}^{0}\|$ and $\|\nabla (u_{j}^{0}-u_{j, h}^{0})\|$ are both at least $O(h^2)$ accurate. Then, the approximation error of the ensemble scheme \eqref{Second-Order-h}  at time $t_N$ satisfies
\begin{equation}
 \frac{1}{4}\|u_j^{N}-u_{j, h}^{N}\|^{2}+ 2C_0\overline{\nu}\Delta t \|\nabla \left(u_{j}^{N}-u_{j, h}^{N}\right)\|^{2}
 \sim \mathcal{O}(h^4+\Delta t^4+h\Delta t^3).
\end{equation}
\end{corollary}

%====================================
\section{Numerical Experiments\label{sec:num}}
%====================================
The goal of this section is two-fold: (i) to numerically illustrate the convergence rate of the ensemble algorithm \eqref{Second-Order-h}, that is, we show it is second-order accurate in time; (ii) to check the stability of the algorithm, in particular, we show that the stability condition \eqref{ineq:CFL-h2} is sharp.  
\subsection{Convergence Test}
We illustrate the convergence rate of \eqref{Second-Order-h} by considering a test problem for the NSE from \cite{GQ98}, which has an analytical solution. 
The problem preserves spatial patterns of the Green-Taylor
solution \cite{B05,GT37} but the vortices do not decay as $t\rightarrow \infty$. 
On the unit square $\Omega = [0, 1]^2$, we define
\begin{align*}
u_{ref}& =[-s(t)\cos x\sin y, s(t)\sin x\cos y]^\top, \\
p_{ref}& =-\frac{1}{4}[\cos (2x)+\cos (2y)]s^{2}(t),
\end{align*}
with $s(t)=\sin (2t)$ and the corresponding source term is
$$f_{ref}(x,y,t)=\left(s^{\prime }(t)+2\nu s(t)\right)[-\cos x\sin y, \sin x\cos y]^\top.$$ 
The initial condition $u_{ref}^0(x, y) = [0, 0]^\top$ and $u_{ref}^0(x, y)$ satisfies inhomogeneous Dirichlet boundary conditions. 
%\textcolor{red}{Nan: I know you checked BC, but I would like to double check with you again here, since the domain is a unit square, we don't have homogeneous BC in our codes, right?} {\color{blue} yes, it's inhomogeneous and equal to the exact solution for each different ensemble member.}
%${u}=u_{true}\text{ on }\partial\Omega.$

To check the convergence, we consider an ensemble of two members with distinct viscosity coefficients and perturbed initial conditions. 
For the first member, the viscosity coefficient $\nu_1=0.2$ and the exact solution is chosen as $u_1=(1+\epsilon)u_{ref}$ whereas for the second member we $\nu_2=0.3$ and $u_2= (1-\epsilon)u_{ref}$, where $\epsilon= 10^{-3}$. 
%The initial conditions of the two simulations are perturbed in the same way as used in \cite{JL14}. 
%The generation of perturbations to initial conditions and source terms is application dependable. In this simple test, we generate perturbations to initial conditions in the same way as in \cite{JL14}. 
The initial conditions, boundary conditions, and the source terms are adjusted accordingly. 
%which are the solutions to NSE corresponding to two different initial conditions. 

For this choice of parameters, we have $\vert \nu_j -\overline{\nu}\vert/\overline{\nu} =\frac{1}{5}$ for both $j= 1$ and $j=2$, hence the stability condition \eqref{ineq:CFL-h2} is satisfied. 
We first apply the ensemble algorithm \eqref{Second-Order-h} with the $P^2$-$P^1$ Taylor-Hood FE and evaluate the rate of convergence. 
The initial mesh size and time step size are chosen to be $h=0.1$ and $\Delta t=0.05$; both the spatial and temporal discretization are uniformly refined. 
The numerical results are listed in Table \ref{tab:t5ensemble} for which
$$
\|\mathcal{E}^{E}_j\|_{\infty, 0}= \max_{0\leq n\leq N} \|u_j^n - u_{j, h}^n\|
\text{\qquad and \qquad}
\|\nabla \mathcal{E}^{E}_j\|_{2, 0}= \sqrt{\Delta t \sum_{n=0}^{N} \|u_j^n - u_{j, h}^n\|^2}.
$$
It is seen that the convergence rates for both $u_{1}$ and $u_{2}$ are second order, which matches our theoretical analysis. 
% ensemble simulations
\begin{table}[htp]
\centering
{\small
\caption{\noindent Approximation errors for ensemble simulations of two members with inputs $\nu_1= 0.2$, $u_{1, 0}= (1+10^{-3})u^0$ and $\nu_2= 0.3$, $u_{2, 0}= (1-10^{-3})u^0$.}
\label{tab:t5ensemble}%
\begin{tabular}{|c||c|c||c|c||c|c||c|c|}
\hline
$1/h$& $\|\mathcal{E}^{E}_1\|_{\infty, 0}$ & rate & $\|\nabla \mathcal{E}^{E}_1\|_{2,0}$ & rate & $\|\mathcal{E}^{E}_2\|_{\infty, 0}$ & rate & $\|\nabla \mathcal{E}^{E}_2\|_{2,0}$ & rate\\
\hline
10  & $1.02e-04$ &   --  & $8.51e-04$  &   -- & $8.02e-05$ &   --  & $7.99e-04$  &   -- \\
20  & $2.60e-05$ &   1.98  & $2.12e-04$  &   2.00 & $2.03e-05$ &   1.98  & $1.99e-04$  &   2.00 \\
40  & $6.54e-06$ &   1.99  & $5.31e-05$  &   2.00 & $5.12e-06$ &   1.99  & $4.99e-05$  &   2.00 \\
80  & $1.64e-06$ &   1.99  & $1.33e-05$  &   2.00 & $1.28e-06$ &   2.00  & $1.25e-05$  &   2.00 \\
\hline
\end{tabular}
}
\end{table}

Furthermore, we implement the two individual simulations separately. 
Comparing the ensemble simulation solutions in Table \ref{tab:t5ensemble} with the independent simulation results listed in Table \ref{tab:t5individual}, we observe that the former achieves the same order of accuracy as the latter. 
% individual simulations
\begin{table}[htp]
\centering
{\small
\caption{\noindent Approximation errors for two individual simulations: $\nu_1=0.2$, $u_{1, 0}= (1+10^{-3})u^0$,  and $\nu_2=0.3$, $u_{2, 0}= (1-10^{-3})u^0$.}
\label{tab:t5individual}%
\begin{tabular}{|c||c|c||c|c||c|c||c|c|}
\hline
$1/h$& $\|\mathcal{E}^{S}_1\|_{\infty, 0}$ & rate & $\|\nabla \mathcal{E}^{S}_1\|_{2,0}$ & rate & $\|\mathcal{E}^{S}_2\|_{\infty, 0}$ & rate & $\|\nabla \mathcal{E}^{S}_2\|_{2,0}$ & rate\\
\hline
10  & $1.08e-04$ &   --       & $8.79e-04$  &   --     & $7.64e-05$ &   --       & $7.79e-04$  &   -- \\
20  & $2.74e-05$ &   1.98  & $2.20e-04$  &   2.00 & $1.94e-05$ &   1.98  & $1.94e-04$  &   2.00 \\
40  & $6.92e-06$ &   1.99  & $5.50e-05$  &   2.00 & $4.87e-06$ &   1.99  & $4.85e-05$  &   2.00 \\
80  & $1.74e-06$ &   1.99  & $1.38e-05$  &   1.99 & $1.22e-06$ &   2.00  & $1.21e-05$  &   2.00 \\
\hline
\end{tabular}
}
\end{table}

\subsection{Stability tests}
Next, we check the stability of our algorithm by considering the problem of a flow between two offset circles \cite{JL14,J15,JL15,JKL15}. The domain is a disk with a smaller off-center
obstacle inside. Letting $r_{1}=1$, $r_{2}=0.1$, and $c=(c_{1},c_{2})=(\frac{1}{2}
,0)$, the domain is given by
\[
\Omega=\{(x,y)\,\,:\,\,x^{2}+y^{2}\leq r_{1}^{2} \,\,\text{ and }\,\, (x-c_{1})^{2}%
+(y-c_{2})^{2}\geq r_{2}^{2}\}.
\]
\begin{figure}[h]
\begin{center}
\includegraphics[width=.6\textwidth]{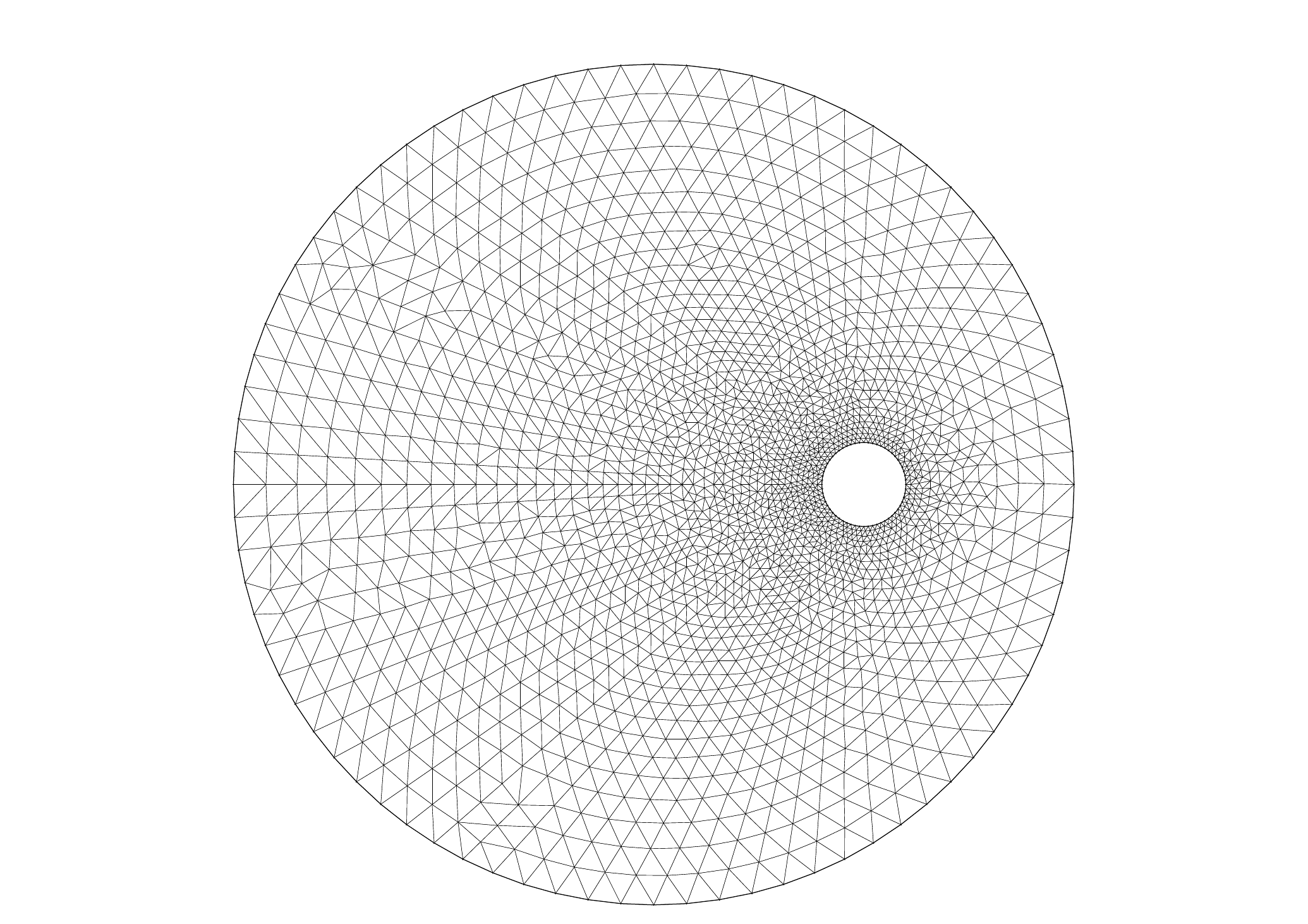}
\end{center}
\caption{Mesh for the flow between two offset cylinders example.}
\label{mesh}
\end{figure}
\noindent The flow is driven by a counterclockwise rotational body force
\[
f(x,y,t)=[-6y(1-x^{2}-y^{2}), 6x(1-x^{2}-y^{2}) ]^\top
\]
with no-slip boundary conditions imposed on both circles. 
%The flow between the two circles shows interesting structures interacting with the inner circle. 
A von K$\acute{a}$rm$\acute{a}$n vortex street forms behind the inner circle and then re-interacts with that circle and with itself, generating complex flow patterns.
We consider multiple numerical simulations of the flow with different viscosity coefficients using the ensemble-based algorithm \eqref{Second-Order-h}. 
For spatial discretization, we apply the $P^2$-$P^1$ Taylor-Hood element pair on a triangular mesh that is generated by Delaunay triangulation with $80$ mesh points on the outer circle and $60$ mesh points on the inner circle and with refinement near the inner circle, resulting in $18,638$ degrees of freedom; see Figure \ref{mesh}.

In order to illustrate the stability analysis, we choose three sets of viscosity coefficients:
\begin{itemize}
\item[] Case 1: \quad $\nu_1=0.021$,\,\, $\nu_2=0.030$,\,\, $\nu_3=0.039$;
\item[] Case 2: \quad $\nu_1=0.019$,\,\, $\nu_2=0.030$,\,\, $\nu_3=0.041$;
\item[] Case 3: \quad $\nu_1=0.015$,\,\, $\nu_2=0.0394$,\,\, $\nu_3=0.0356$.
\end{itemize}
The average of the viscosity coefficients is $\overline{\nu}= 0.03$ for all the cases. 
However, the stability condition \eqref{ineq:CFL-h2} does not hold except the first case because 
\begin{itemize}
\item[] Case 1: \quad $\frac{\vert\nu_1-\overline{\nu}\vert}{\overline{\nu}}=\frac{3}{10}$,\,\, $\frac{\vert\nu_2-\overline{\nu}\vert}{\overline{\nu}}=0$,\,\,
$\frac{\vert\nu_3-\overline{\nu}\vert}{\overline{\nu}}=\frac{3}{10}$;
\item[] Case 2: \quad   $\frac{\vert\nu_1-\overline{\nu}\vert}{\overline{\nu}}=\frac{11}{30}$,\,\, $\frac{\vert\nu_2-\overline{\nu}\vert}{\overline{\nu}}=0$,\,\,
$\frac{\vert\nu_3-\overline{\nu}\vert}{\overline{\nu}}=\frac{11}{30}$;
\item[] Case 3: \quad   $\frac{\vert\nu_1-\overline{\nu}\vert}{\overline{\nu}}=\frac{1}{2}$,\,\, $\frac{\vert\nu_2-\overline{\nu}\vert}{\overline{\nu}}=\frac{47}{150}$,\,\,
$\frac{\vert\nu_3-\overline{\nu}\vert}{\overline{\nu}}=\frac{14}{75}$.
\end{itemize}
Among them, the first and third members of Case 2 and the first member of Case 3 have perturbation ratios greater than $\frac{1}{3}$. 
Simulations of all the cases are subject to the same initial condition, boundary condition and body forces for all ensemble members. 
In particular, the initial condition is generated by solving the steady Stokes problem with viscosity $\nu=0.03$ and the same body force $f(x, y, t)$. 
All the simulations are run over the time interval $[0, 5]$ with a time step size $\Delta t= 0.01$. 
For the stability test, we use the kinetic energy as a criterion and compare the ensemble simulation results with independent simulations using the same mesh and time-step size. 

The comparison of the energy evolution of ensemble-based simulations with the corresponding independent simulations is shown in Figures \ref{egy_1}, \ref{egy_2} and \ref{egy_3}. 
It is seen that, for Case 1, the ensemble simulation is stable, but for Cases 2 and 3 it becomes unstable. 
This phenomena coincides with our stability analysis because the condition \eqref{ineq:CFL-h2} holds for all members of Case 1, but does not hold for Cases 2 and 3. 
Indeed, it is observed from Figure \ref{egy_2} that the energy of the third member in Case 2 blows up after $t=1.95$, then affects the other two members and results in their  energy dramatically increasing after $t=2.45$. 
In Case 3, the first member blows up after $t= 4.35$, which influences the other two members and leads to their energy blowing up at $t= 4.9$. 

\begin{figure}[h!]
\begin{center}
\includegraphics[width=.8\textwidth]{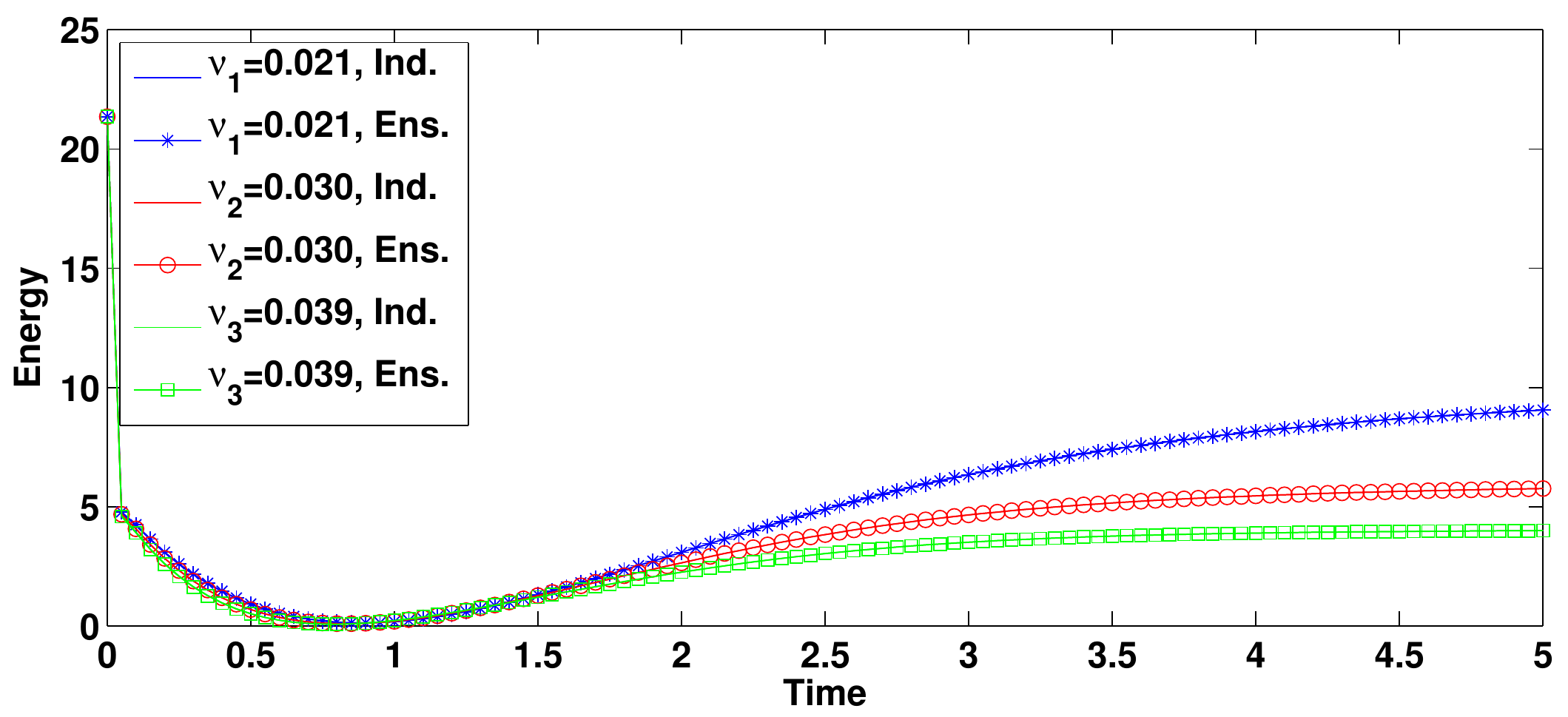}
\end{center}
\caption{For the flow between two offset cylinders, Case 1, the energy evolution of the ensemble (Ens.) and independent simulations (Ind.).}
\label{egy_1}
\end{figure}
\begin{figure}[h!]
\begin{center}
\includegraphics[width=.8\textwidth]{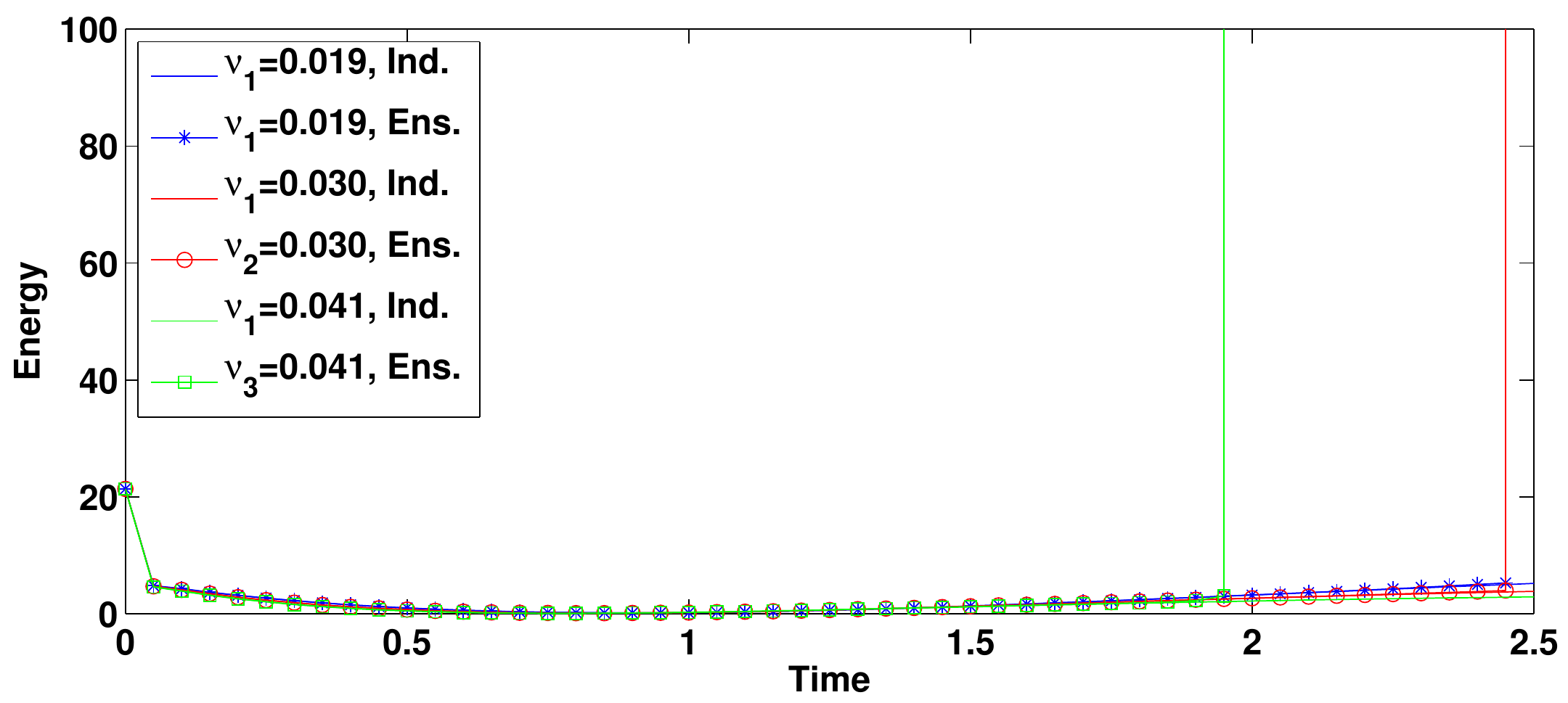}
\end{center}
\caption{For the flow between two offset cylinders, Case 2, the energy evolution of the ensemble (Ens.) and independent simulations (Ind.).}
\label{egy_2}
\end{figure}
\begin{figure}[h!]
\begin{center}
\includegraphics[width=.8\textwidth]{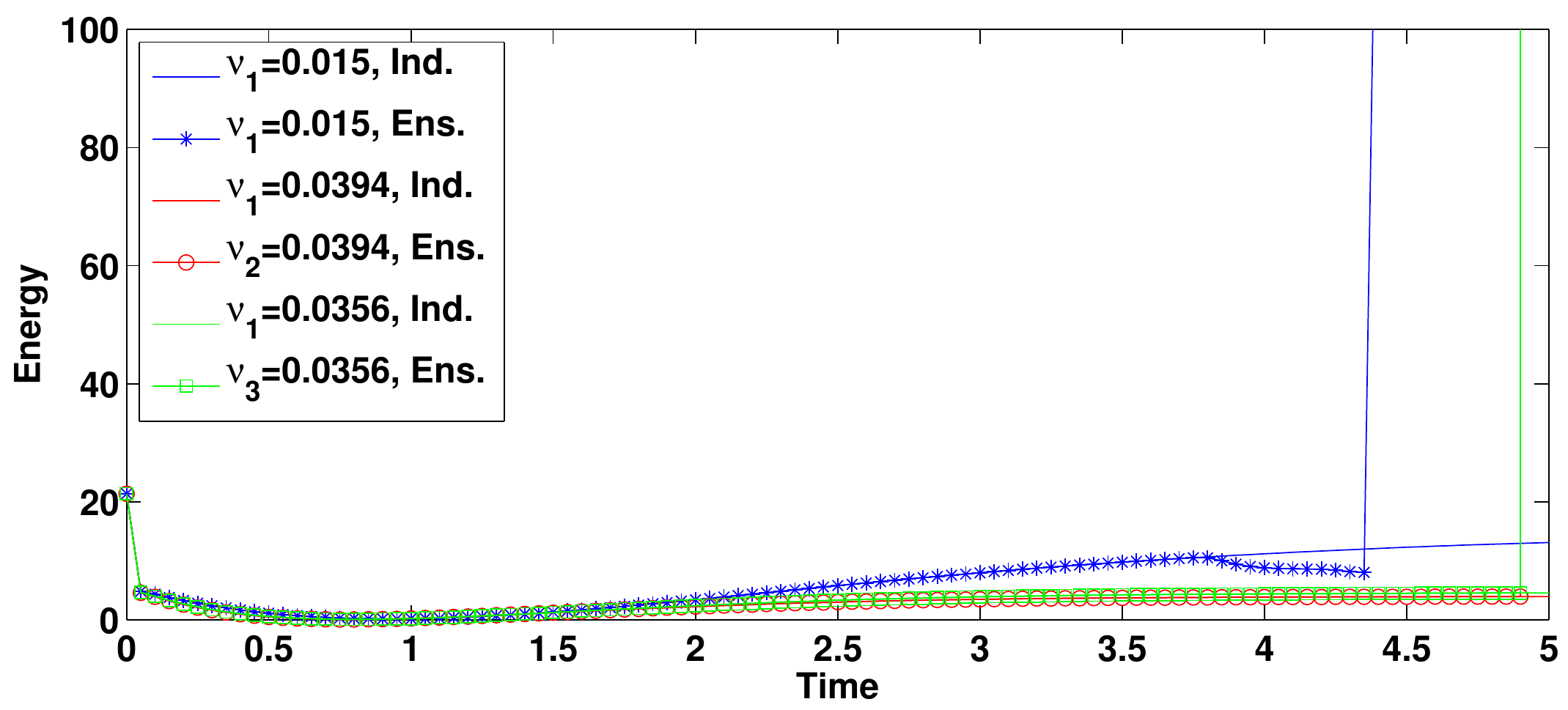}
\end{center}
\caption{For the flow between two offset cylinders, Case 3, the energy evolution of the ensemble (Ens.) and independent simulations (Ind.).}
\label{egy_3}
\end{figure}

\section{Conclusions}
In this paper, we develop a second-order time-stepping ensemble scheme to compute a set of Navier-Stokes equations in which every member is subject to an independent computational setting including a distinct viscosity coefficient, initial condition data, boundary condition data, and/or body force. 
By using the ensemble algorithm, all ensemble members share a common coefficient matrix after  discretization, although with different RHS vectors. 
Therefore, many efficient block iterative solvers such as the block CG and block GMRES can be applied to solve such a single linear system with multiple RHS vectors, 
leading to great savings in both storage and simulation time. 
A rigorous analysis shows the proposed algorithm is conditionally, nonlinearly and long-term stable under two explicit conditions and is second-order accurate in time. 
Two numerical experiments are presented that illustrate our theoretical analysis.
In particular, the first is a test problem having an analytic solution which illustrates that the rate of convergence with respect to time step size is indeed second order, whereas  
the second example is for a flow between two offset cylinders and shows that the stability condition is sharp. 
For future work, we plan to investigate the performance of the ensemble algorithm in data assimilation applications.

\appendix

\section{Proof of Theorem \ref{th:stability}}\label{proofstab}

\begin{proof}
%Under the assumption of the LBB$_h$ condition, we restrict $v_h$ in \eqref{First-Order-h} to the space of discretely divergence free functions $V^h$.     
Setting $v_{h}=u_{j,h}^{n+1}$ and $q_h=p_{j,h}^{n+1}$ in (\ref{Second-Order-h}) and multiplying the result by $\Delta t$ gives
\begin{align*}
&\frac{1}{4}\left(\Vert u_{j,h}^{n+1}\Vert^{2}+\Vert 2u_{j,h}^{n+1}-u_{j,h}^{n}\Vert^2\right)-\frac{1}{4}\left(\Vert u_{j,h}^{n}\Vert^{2}+\Vert 2u_{j,h}^{n}-u_{j,h}^{n-1}\Vert^2\right)\\
&\,\,+\frac{1}{4}\Vert u_{j,h}^{n+1}-2u_{j,h}^{n}+u_{j,h}^{n-1}\Vert^{2}+
\overline{\nu}\Delta t \Vert \nabla u_{j,h}^{n+1}\Vert^{2}\\
&\,\,+ \Delta t b^{*}\left(  2u_{j,h}^{ n}-u_{j,h}^{n-1}-\overline{u}^n_{h},  2u_{j,h}^{n}-u_{j,h}^{n-1},u_{j,h}^{n+1}\right) \\
&= \Delta t \left(f_{j}^{n+1}, u_{j,h}
^{n+1}\right)-(\nu_j-\overline{\nu})\Delta t \left(\nabla (2u_{j,h}^n-u_{j,h}^{n-1}), \nabla u_{j,h}^{n+1}\right).\nonumber
\end{align*}
Applying Young's inequality to the terms on the RHS yields, for any $\alpha, \beta_1, \beta_2 >0$,
\begin{equation}
\label{Young}
\begin{split}
&\frac{1}{4}\left(\Vert u_{j,h}^{n+1}\Vert^{2}+\Vert 2u_{j,h}^{n+1}-u_{j,h}^{n}\Vert^2\right)-\frac{1}{4}\left(\Vert u_{j,h}^{n}\Vert^{2}+\Vert 2u_{j,h}^{n}-u_{j,h}^{n-1}\Vert^2\right)\\
&\,\,+\frac{1}{4}\Vert u_{j,h}^{n+1}-2u_{j,h}^{n}+u_{j,h}^{n-1}\Vert^{2}+
\overline{\nu}\Delta t \Vert \nabla u_{j,h}^{n+1}\Vert^{2}\\
&\,\,+ \Delta t b^{*}\left(  2u_{j,h}^{ n}-u_{j,h}^{n-1}-\overline{u}^n_{h},  2u_{j,h}^{n}-u_{j,h}^{n-1},u_{j,h}^{n+1}\right)\\
&\leq\frac{\alpha\overline{\nu}\Delta t}{4} \Vert\nabla u_{j,h}^{n+1} \Vert^{2}+ \frac{\Delta
t}{ \alpha\overline{\nu}} \Vert f_{j}^{n+1}\Vert_{-1}^{2} +\beta_1\overline{\nu}\Delta t \Vert\nabla u_{j,h}^{n+1} \Vert^{2}\\
&\,\,+\frac{(\nu_j-\overline{\nu})^2\Delta t}{\beta_1\overline{\nu}}\Vert \nabla u_{j,h}^n\Vert^2+\frac{\beta_2\overline{\nu}\Delta t}{4} \Vert\nabla u_{j,h}^{n+1} \Vert^{2}+\frac{(\nu_j-\overline{\nu})^2\Delta t}{\beta_2\overline{\nu}}\Vert \nabla u_{j,h}^{n-1}\Vert^2. 
\end{split}
\end{equation}
Because the last four terms on the RHS of \eqref{Young} need to be absorbed into $\overline{\nu}\Delta t \Vert\nabla u_{j,h}^{n+1}\Vert^{2}$ on the LHS, we minimize $\beta_1\overline{\nu}\Delta t \Vert\nabla u_{j,h}^{n+1} \Vert^{2}+\frac{(\nu_j-\overline{\nu})^2\Delta t}{\beta_1\overline{\nu}}\Vert \nabla u_{j,h}^n\Vert^2$ by taking $\beta_1 = \frac{\vert \nu_j-\overline{\nu}\vert }{\overline{\nu} }$ and $\frac{\beta_2\overline{\nu}\Delta t}{4} \Vert\nabla u_{j,h}^{n+1} \Vert^{2}+\frac{(\nu_j-\overline{\nu})^2\Delta t}{\beta_2\overline{\nu}}\Vert \nabla u_{j,h}^{n-1}\Vert^2$ by taking $\beta_2 = \frac{2\vert \nu_j-\overline{\nu}\vert }{\overline{\nu} }$. 
Then \eqref{Young} becomes
\begin{equation}
\label{ineq:s0}
\begin{split}
&\frac{1}{4}\left(\Vert u_{j,h}^{n+1}\Vert^{2}+\Vert 2u_{j,h}^{n+1}-u_{j,h}^{n}\Vert^2\right)-\frac{1}{4}\left(\Vert u_{j,h}^{n}\Vert^{2}+\Vert 2u_{j,h}^{n}-u_{j,h}^{n-1}\Vert^2\right)\\
&\quad+\frac{1}%
{4}\Vert u_{j,h}^{n+1}-2u_{j,h}^{n}+u_{j,h}^{n-1}\Vert^{2}+
\overline{\nu}\Delta t \Vert \nabla u_{j,h}^{n+1}\Vert^{2}\\
&\quad+ \Delta t b^{*}\left(  2u_{j,h}^{ n}-u_{j,h}^{n-1}-\overline{u}^n_{h},  2u_{j,h}^{n}-u_{j,h}^{n-1},u_{j,h}^{n+1}\right)\\
&\leq\frac{\alpha\overline{\nu}\Delta t}{4} \Vert\nabla u_{j,h}^{n+1} \Vert^{2}+ \frac{\Delta
t}{ \alpha\overline{\nu}} \Vert f_{j}^{n+1}\Vert_{-1}^{2}+\frac{3\vert \nu_j-\overline{\nu}\vert\Delta t}{2} \Vert\nabla u_{j,h}^{n+1} \Vert^{2}\\
&\quad+\vert \nu_j-\overline{\nu}\vert \Delta t\Vert \nabla u_{j,h}^n\Vert^2+\frac{\vert \nu_j-\overline{\nu}\vert \Delta t}{2}\Vert \nabla u_{j,h}^{n-1}\Vert^2.
\end{split}
\end{equation}

Next, we bound the trilinear term using the inequality \eqref{In2} and
the inverse inequality \eqref{inverse}:
\begin{equation*}
\begin{split}
&\quad b^{*}\left(2 u_{j,h}^{ n}-u_{j,h}^{n-1}-\overline{u}_h^n,  2u_{j,h}^{n}-u_{j,h}^{n-1},u_{j,h}^{n+1}\right)\\
&= b^{*}\left( 2 u_{j,h}^{ n}-u_{j,h}^{n-1}-\overline{u}_h^n,  -u_{j,h}^{n+1}+2u_{j,h}^{n}-u_{j,h}^{n-1},u_{j,h}^{n+1}\right)\\
%=\Delta t b^{*}\left(  u_{j,h}^{\prime n}, u_{j,h}^{n+1}, u_{j,h}^{n+1}-2u_{j,h}^{n}+u_{j,h}^{n-1}\right)\nonumber\\
&\leq C\Vert\nabla (2 u_{j,h}^{ n}-u_{j,h}^{n-1}-\overline{u}_h^n)\Vert\Vert\nabla u_{j,h}^{n+1}\Vert\Vert \nabla (u_{j,h}^{n+1}-2u_{j,h}^{n}+u_{j,h}^{n-1})\Vert^{\frac{1}{2}}\\
&\qquad\Vert u_{j,h}^{n+1}-2u_{j,h}^{n}+u_{j,h}^{n-1}\Vert^{\frac{1}{2}}\\
& \leq C h^{-\frac{1}{2}%
}\Vert\nabla (2 u_{j,h}^{ n}-u_{j,h}^{n-1}-\overline{u}_h^n)\Vert\Vert\nabla u_{j,h}^{n+1}\Vert\Vert u_{j,h}^{n+1}-2u_{j,h}^{n}+u_{j,h}^{n-1}\Vert\text{ .}
\end{split}
\end{equation*}
Using Young's inequality again gives 
\begin{equation}
\label{ineq:s1}
\begin{split}
&\Delta t \Big\vert b^{*}\left(  2 u_{j,h}^{ n}-u_{j,h}^{n-1}-\overline{u}_h^n,  2u_{j,h}^{n}-u_{j,h}^{n-1},u_{j,h}^{n+1}\right)\Big\vert\\
&\leq  C \frac{\Delta t^{2}}{h}
\Vert \nabla (2 u_{j,h}^{ n}-u_{j,h}^{n-1}-\overline{u}_h^n)\Vert^{2} \Vert\nabla u_{j,h}^{n+1}\Vert^{2}
+
\frac{1}{8}\Vert u_{j,h}^{n+1}-2u_{j,h}^{n}+u_{j,h}^{n-1}\Vert^{2}\text{ .}
\end{split}
\end{equation}
Substituting \eqref{ineq:s1} into \eqref{ineq:s0} and combining like terms, we have 
\begin{equation}
\begin{split}
&\frac{1}{4}\left(\Vert u_{j,h}^{n+1}\Vert^{2}+\Vert 2u_{j,h}^{n+1}-u_{j,h}^{n}\Vert^2\right)-\frac{1}{4}\left(\Vert u_{j,h}^{n}\Vert^{2}+\Vert 2u_{j,h}^{n}-u_{j,h}^{n-1}\Vert^2\right)\\
&\,\,+\frac{1}{8}\Vert u_{j,h}^{n+1}-2u_{j,h}^{n}+u_{j,h}^{n-1}\Vert^{2}+ \overline{\nu}\Delta
t \left(1-\frac{\alpha}{4}-\frac{3\vert \nu_j-\overline{\nu}\vert}{2\overline{\nu}}\right)\Vert\nabla u_{j,h}^{n+1}\Vert^{2}\\
&\leq\frac{\Delta t}{\alpha\overline{\nu}} \Vert f_{j}^{n+1}\Vert_{-1}^{2} + C \frac{\Delta
t^{2}}{h} \Vert\nabla( 2 u_{j,h}^{ n}-u_{j,h}^{n-1}-\overline{u}_h^n)\Vert^{2} \Vert\nabla
u_{j,h}^{n+1}\Vert^{2}\\
&\,\,+\vert \nu_j-\overline{\nu}\vert \Delta t\Vert \nabla u_{j,h}^n\Vert^2+\frac{\vert \nu_j-\overline{\nu}\vert \Delta t}{2}\Vert \nabla u_{j,h}^{n-1}\Vert^2.
\end{split}
\end{equation}
For any $0<\sigma <1$, 
\begin{equation}
\label{ineq:stable01}
\begin{split}
&\frac{1}{4}\left(\Vert u_{j,h}^{n+1}\Vert^{2}+\Vert 2u_{j,h}^{n+1}-u_{j,h}^{n}\Vert^2\right)-\frac{1}{4}\left(\Vert u_{j,h}^{n}\Vert^{2}+\Vert 2u_{j,h}^{n}-u_{j,h}^{n-1}\Vert^2\right)\\
&\,\,+\frac{1}{8}\Vert u_{j,h}^{n+1}-2u_{j,h}^{n}+u_{j,h}^{n-1}\Vert^{2}
+ \overline{\nu}\Delta t \sigma\left(1-\frac{\alpha}{4}-\frac{3\vert \nu_j-\overline{\nu}\vert}{2\overline{\nu}}\right)\left(\Vert\nabla u_{j,h}^{n+1}\Vert^{2}-\Vert\nabla u_{j,h}^{n}\Vert^{2}\right)\\
&\,\,+\overline{\nu}\Delta t \left((1-\sigma)\left(1-\frac{\alpha}{4}-\frac{3\vert \nu_j-\overline{\nu}\vert}{2 \overline{\nu}}\right)-\frac{C \Delta t}{\overline{\nu} h}\Vert\nabla(u_{j,h}^{n}-\overline{u}_{h}^{n})\Vert^{2}\right) \Vert\nabla
u_{j,h}^{n+1}\Vert^{2}\\
&\,\,+\overline{\nu} \Delta t \left(\frac{2}{3}\sigma\left(1-\frac{\alpha}{4}-\frac{3\vert \nu_j-\overline{\nu}\vert}{2 \overline{\nu}}\right)-\frac{\vert \nu_j-\overline{\nu}\vert}{\overline{\nu}}\right)\Vert \nabla u_{j,h}^n\Vert^2\\
&\,\,+\overline{\nu} \Delta t \left(\frac{1}{3}\sigma\left(1-\frac{\alpha}{4}-\frac{3\vert \nu_j-\overline{\nu}\vert}{2 \overline{\nu}}\right)-\frac{\vert \nu_j-\overline{\nu}\vert}{2\overline{\nu}}\right) 
\left( \Vert \nabla u_{j,h}^n\Vert^2-\Vert \nabla u_{j,h}^{n-1}\Vert^2\right)\\
&\leq\frac{\Delta t}{ \alpha\overline{\nu}} \Vert f_{j}^{n+1}\Vert
_{-1}^{2}\text{ .}
\end{split}
\end{equation}
Because $\alpha>0$ is arbitrary, we take $\alpha=4-\frac{2(\sigma+1)}{\sigma}\sqrt{\mu}$. To make sure that $\alpha$ is greater than $0$, we need
\begin{align*}
\sigma >\frac{\sqrt{\mu}}{2-\sqrt{\mu}}, \text{ where} \quad \frac{\sqrt{\mu}}{2-\sqrt{\mu}}\in (0,1) .
%\text{ since } \sqrt{\mu}\in (0,1).
\end{align*}
%Now taking $\sigma= \frac{\sqrt{\mu}}{2-\sqrt{\mu}} + \frac{1}{2}(1-\frac{\sqrt{\mu}}{2-\sqrt{\mu}})= \frac{1}{2}+\frac{\sqrt{\mu}}{4-2\sqrt{\mu}} = \frac{1}{2-\sqrt{\mu}}\in (\frac{1}{2},1)$, 
Now taking $\sigma= \frac{\sqrt{\mu}+\epsilon}{2-\sqrt{\mu}}$ and $\epsilon \in (0, 2-2\sqrt{\mu})$ , 
\eqref{ineq:stable01} becomes
\begin{equation}
\label{ineq:stable001}
\begin{aligned}
&\frac{1}{4}\left(\Vert u_{j,h}^{n+1}\Vert^{2}+\Vert 2u_{j,h}^{n+1}-u_{j,h}^{n}\Vert^2\right)-\frac{1}{4}\left(\Vert u_{j,h}^{n}\Vert^{2}+\Vert 2u_{j,h}^{n}-u_{j,h}^{n-1}\Vert^2\right)\\
&\,\,+\frac{1}{8}\Vert u_{j,h}^{n+1}-2u_{j,h}^{n}+u_{j,h}^{n-1}\Vert^{2}\\
&+ \overline{\nu}\Delta t \sigma\left(\frac{\sigma+1}{2\sigma}\sqrt{\mu}-\frac{3\vert \nu_j-\overline{\nu}\vert}{2 \overline{\nu}}\right)\left(\Vert\nabla u_{j,h}^{n+1}\Vert^{2}-\Vert\nabla u_{j,h}^{n}\Vert
^{2}\right)\\
&+\overline{\nu}\Delta t \left( (1-\sigma)\left(\frac{\sigma+1}{2\sigma}\sqrt{\mu}-\frac{3\vert \nu_j-\overline{\nu}\vert}{2 \overline{\nu}}\right)-\frac{C \Delta
t}{\overline{\nu} h}\Vert\nabla(u_{j,h}^{n}-\overline{u}_{h}^{n})\Vert^{2}\right) \Vert\nabla
u_{j,h}^{n+1}\Vert^{2}\\
&+\overline{\nu} \Delta t \left( (\sigma+1)\left(\frac{\sqrt{\mu}}{3}-\frac{\vert \nu_j-\overline{\nu}\vert}{\overline{\nu}}\right)\right)\Vert \nabla u_{j,h}^n\Vert^2\\
&+\overline{\nu} \Delta t \left( \frac{(\sigma+1)}{2}\left(\frac{\sqrt{\mu}}{3}-\frac{\vert \nu_j-\overline{\nu}\vert}{\overline{\nu}}\right)\right)\left(\Vert \nabla u_{j,h}^n\Vert^2-\Vert \nabla u_{j,h}^{n-1}\Vert^2\right)\leq\frac{\Delta t}{ \alpha\overline{\nu}} \Vert f_{j}^{n+1}\Vert
_{-1}^{2}\text{ .}
\end{aligned}
\end{equation}
Stability follows if the following conditions hold:
\begin{gather}
\frac{\sigma+1}{2\sigma}\sqrt{\mu}-\frac{3\vert \nu_j-\overline{\nu}\vert}{2 \overline{\nu}}\geq 0,\\
(1-\sigma)\left(\frac{\sigma+1}{2\sigma}\sqrt{\mu}-\frac{3\vert \nu_j-\overline{\nu}\vert}{2 \overline{\nu}}\right)-\frac{C \Delta
t}{\overline{\nu} h}\Vert\nabla(u_{j,h}^{n}-\overline{u}_{h}^{n})\Vert^{2}\geq0,\label{cond1}\\
\quad\text{and}\quad  \frac{\sqrt{\mu}}{3}-\frac{\vert \nu_j-\overline{\nu}\vert}{\overline{\nu}}\geq 0.\label{cond2}
\end{gather}
Under the assumption of (\ref{ineq:CFL-h2}), we have
\begin{gather*}
\frac{\sqrt{\mu}}{3}-\frac{\vert \nu_j-\overline{\nu}\vert}{\overline{\nu}}
 \geq 0 \quad \text{ and } \\
\frac{\sigma+1}{2\sigma}\sqrt{\mu}-\frac{3\vert \nu_j-\overline{\nu}\vert}{2 \overline{\nu}}=\frac{\sqrt{\mu}(2-\sqrt{\mu})}{2(\sqrt{\mu}+\epsilon)}\geq 0.
\end{gather*}
Together with the first assumption in \eqref{ineq:CFL-h1}, we have
\begin{gather*}
(1-\sigma)\left(\frac{\sigma+1}{2\sigma}\sqrt{\mu}-\frac{3\vert \nu_j-\overline{\nu}\vert}{2 \overline{\nu}}\right)-\frac{C \Delta
t}{\overline{\nu} h}\Vert\nabla(u_{j,h}^{n}-\overline{u}_{h}^{n})\Vert^{2} \\
\geq \frac{(2-2\sqrt{\mu}-\epsilon)\sqrt{\mu}}{2(\sqrt{\mu}+\epsilon)}-\frac{C \Delta t}{\overline{\nu} h}\Vert\nabla(u_{j,h}^{n}
-\overline{u}_{h}^{n})\Vert^{2} \\
\geq\frac{(2-2\sqrt{\mu}-\epsilon)\sqrt{\mu}}{2(\sqrt{\mu}+\epsilon)}-\frac{(2-2\sqrt{\mu}-\epsilon)\sqrt{\mu}}{2(\sqrt{\mu}+\epsilon)}=0.
\end{gather*}
Therefore, assuming both conditions \eqref{ineq:CFL-h1}-\eqref{ineq:CFL-h2} hold,  (\ref{ineq:stable001}) reduces to
\begin{equation}
\label{ineq:stable02}
\begin{split}
&\frac{1}{4}\left(\Vert u_{j,h}^{n+1}\Vert^{2}+\Vert 2u_{j,h}^{n+1}-u_{j,h}^{n}\Vert^2\right)-\frac{1}{4}\left(\Vert u_{j,h}^{n}\Vert^{2}+\Vert 2u_{j,h}^{n}-u_{j,h}^{n-1}\Vert^2\right)\\
&\,\,+\frac{1}%
{8}\Vert u_{j,h}^{n+1}-2u_{j,h}^{n}+u_{j,h}^{n-1}\Vert^{2}\\
&\,\,+\overline{\nu}\Delta t  \frac{\sqrt{\mu}+\epsilon}{2-\sqrt{\mu}}\left(\frac{\sqrt{\mu}}{2}\frac{2+\epsilon}{\sqrt{\mu}+\epsilon}-\frac{3\vert \nu_j-\overline{\nu}\vert}{2 \overline{\nu}}\right)
\left(\Vert\nabla u_{j,h}^{n+1}\Vert^{2}-\Vert\nabla u_{j,h}^{n}\Vert^{2}\right)  \\
&\leq\frac{\sqrt{\mu}+\epsilon}{2\epsilon(2-\sqrt{\mu})}\frac{\Delta t}{ \overline{\nu}} \Vert f_{j}^{n+1}\Vert_{-1}^{2}\text{
.}
\end{split}
\end{equation}
Summing up (\ref{ineq:stable02}) from $n=1$ to $N-1$ results in
\begin{equation}
\begin{split}
&\frac{1}{4}\left(\Vert u_{j,h}^{N}\Vert^{2}+\Vert 2u_{j,h}^{N}-u_{j,h}^{N-1}\Vert^2\right)+\frac{1}{8}\sum_{n=1}^{N-1}\Vert u_{j,h}^{n+1}-2u_{j,h}^{n}+u_{j,h}^{n-1}\Vert^{2}\\
&\,\,+\overline{\nu}\Delta t  \frac{\sqrt{\mu}+\epsilon}{2-\sqrt{\mu}}\left(\frac{\sqrt{\mu}}{2}\frac{2+\epsilon}{\sqrt{\mu}+\epsilon}-\frac{3\vert \nu_j-\overline{\nu}\vert}{2 \overline{\nu}}\right)
\|\nabla u_{j,h}^{N}\|^{2}\\
&\leq\sum_{n=1}^{N-1}\frac{\sqrt{\mu}+\epsilon}{2\epsilon(2-\sqrt{\mu})}\frac{\Delta t}{\overline{\nu}}\|f_{j}^{n+1}\|_{-1}^{2}+ \frac{1}{4}\left(\Vert u_{j,h}^{1}\Vert^{2}+\Vert 2u_{j,h}^{1}-u_{j,h}^{0}\Vert^2\right)\\
&\,\,+\overline{\nu}\Delta t  \frac{\sqrt{\mu}+\epsilon}{2-\sqrt{\mu}}\left(\frac{\sqrt{\mu}}{2}\frac{2+\epsilon}{\sqrt{\mu}+\epsilon}-\frac{3\vert \nu_j-\overline{\nu}\vert}{2 \overline{\nu}}\right)\|\nabla u_{j,h}^{1}\|^{2}\text{
.}
\end{split}
\label{eq:sta}
\end{equation}
This shows that the ensemble algorithm \eqref{Second-Order-h} is stable under conditions \eqref{ineq:CFL-h1}-\eqref{ineq:CFL-h2}.
\end{proof}

\section{Proof of Lemma \ref{cons1}}\label{prooflem1}

\begin{proof}
To prove \eqref{cons1}, we first rewrite
%\begingroup
\allowdisplaybreaks
\begin{align*}
&3(u^{n+1}-u^n)-(u^n-u^{n-1})-2\Delta t u_t^{n+1}
=3\int_{t^n}^{t^{n+1}}u_t dt-\int_{t^{n-1}}^{t^n}u_t dt -2\Delta t u_t^{n+1}\\
%&=3\int_{t^n}^{t^{n+1}}\frac{d}{dt}(t-t^n) u_t dt-\int_{t^{n-1}}^{t^n}\frac{d}{dt}(t-t^{n-1}) u_t dt 
%-2\Delta t u_t^{n+1}\\
&=3\left(\left[(t-t^n)u_t\right]_{t^n}^{t^{n+1}}-\int_{t^n}^{t^{n+1}}(t-t^n) u_{tt} dt\right)\\
&\quad -\left(\left[(t-t^{n-1})u_t\right]_{t^{n-1}}^{t^{n}}-\int_{t^{n-1}}^{t^{n}}(t-t^{n-1}) u_{tt} dt\right)
-2\Delta t u_t^{n+1}\\
&=3\Delta t u_t^{n+1} - \Delta t u_t^n -2\Delta tu_t^{n+1}-3\int_{t^n}^{t^{n+1}}\frac{d}{dt}\left(\frac{1}{2}(t-t^n)^2\right) u_{tt} dt\\
&\quad +\int_{t^{n-1}}^{t^{n}}\frac{d}{dt}\left(\frac{1}{2}(t-t^{n-1})^2\right) u_{tt} dt\\
%&=\Delta t \int_{t^n}^{t^{n+1}}u_{tt}dt 
%-3\left(\left[\frac{1}{2}(t-t^n)^2u_{tt}\right]_{t^n}^{t^{n+1}}-\int_{t^n}^{t^{n+1}}\frac{1}{2}(t-t^n)^2u_{ttt}dt\right)\\
%&\quad +\left[\frac{1}{2}(t-t^{n-1})^2u_{tt}\right]_{t^{n-1}}^{t^{n}}
%-\int_{t^{n-1}}^{t^{n}}\frac{1}{2}(t-t^{n-1})^2u_{ttt}dt\\
&=\Delta t \left( \left[(t-t^n)u_{tt}\right]_{t^n}^{t^{n+1}}-\int_{t^n}^{t^{n+1}}(t-t^n)u_{ttt}dt\right )
-\frac{3}{2}\Delta t^2 u_{tt}^{n+1}+\frac{1}{2}\Delta t^2u_{tt}^n \\
&\quad+3\int_{t^n}^{t^{n+1}}\frac{1}{2}(t-t^n)^2u_{ttt}dt-\int_{t^{n-1}}^{t^{n}}\frac{1}{2}(t-t^{n-1})^2u_{ttt}dt\\
&=-\frac{1}{2}\Delta t^2 \int_{t^n}^{t^{n+1}}u_{ttt}\,dt
+3\int_{t^n}^{t^{n+1}}\frac{1}{2}(t-t^n)^2u_{ttt}dt-\int_{t^{n-1}}^{t^{n}}\frac{1}{2}(t-t^{n-1})^2u_{ttt}dt
\end{align*}
Then the $L^2$ norm of the term of interest can be estimated as follows
\begin{align*}
&\qquad \Big\Vert\frac{3u^{n+1}-4u^n+u^{n-1}}{2\Delta t}-u_t^{n+1}\Big\Vert^2\\
&=\frac{1}{4\Delta t^2}\bigintsss_{\Omega} \Bigg|-\frac{1}{2}\Delta t^2 \left[\int_{t^n}^{t^{n+1}}u_{ttt}dt\right]
+3\int_{t^n}^{t^{n+1}}\frac{1}{2}(t-t^n)^2u_{ttt}dt\\
&\quad -\int_{t^{n-1}}^{t^{n}}\frac{1}{2}(t-t^{n-1})^2u_{ttt}dt\Bigg |^2 dx\\
&\leq\frac{1}{2\Delta t^2}\bigintsss_{\Omega} \Bigg (\frac{1}{4}\Delta t^4 \Bigg |\int_{t^n}^{t^{n+1}}u_{ttt}dt\Bigg |^2
+\frac{9}{4}\Bigg|\int_{t^n}^{t^{n+1}}(t-t^n)^2u_{ttt}dt\Bigg |^2\\
&\quad +\frac{1}{4}\Bigg|\int_{t^{n-1}}^{t^{n}}(t-t^{n-1})^2u_{ttt}dt\Bigg|^2\Bigg)dx\\
&\leq\frac{1}{2\Delta t^2}\bigintsss_{\Omega} \Bigg (\frac{1}{4}\Delta t^4 \Bigg |\int_{t^n}^{t^{n+1}}u_{ttt}dt\Bigg |^2
+\frac{9}{4}\Delta t^4\Bigg|\int_{t^n}^{t^{n+1}}u_{ttt}dt\Bigg |^2+\frac{1}{4}\Delta t^4\Bigg|\int_{t^{n-1}}^{t^{n}}u_{ttt}dt\Bigg|^2\Bigg)dx\\
&\leq\frac{1}{2\Delta t^2}\bigintsss_{\Omega} \Bigg(\frac{1}{4}\Delta t^5 \int_{t^n}^{t^{n+1}}|u_{ttt}|^2dt
+\frac{9}{4}\Delta t^5 \int_{t^n}^{t^{n+1}}|u_{ttt}|^2dt+\frac{1}{4}\Delta t^5\int_{t^{n-1}}^{t^{n}}|u_{ttt}|^2dt\Bigg)dx\\
&\leq \frac{5}{2}\Delta t^3 \int_{\Omega}\left(\int_{t^{n-1}}^{t^{n+1}}|u_{ttt}|^2 dt\right) dx\leq \frac{5}{2}\Delta t^3\int_{t^{n-1}}^{t^{n+1}}\Vert u_{ttt}\Vert^2 dt.
\end{align*}
%\endgroup
This completes the proof.
\end{proof}

\section{Proof of Theorem \ref{th:err}}\label{proofthm1}

\begin{proof}
The true solution$(u_{j}, p_j)$ of the NSE  satisfies

\begin{gather}
\left(\frac{3u_{j}^{n+1}-4u_{j}^{n}+u_j^{n-1}}{2 \Delta t}, v_{h}\right)+ b^{*}\left(u_{j}^{n+1},
u_{j}^{n+1}, v_{h}\right)+ \nu_j \left(\nabla u_{j}^{n+1}, \nabla v_{h}\right)\label{eq:convtrue}\\
- \left(p_{j}
^{n+1},\nabla\cdot v_{h}\right)
=\left(f_{j}^{n+1}, v_{h}\right)+ \mathrm{Intp}\left(u_{j}^{n+1};v_{h}\right)\text{,}\qquad \text{for all } v_{h}\in V_{h}\text{,}\nonumber
\end{gather}
where $\mathrm{Intp}\left(u_{j}^{n+1};v_{h}\right)$ is defined as
\[
\mathrm{Intp}\left(u_{j}^{n+1};v_{h}\right)=\left(\frac{3u_{j}^{n+1}-4u_{j}^{n}+u_j^{n-1}}{2\Delta t}-u_{j,t}(t^{n+1}),v_{h}\right)\text{ .}
\]

Let
\begin{equation}
e_{j}^{n}=u_{j}^{n}-u_{j,h}^{n}=(u_{j}^{n}-I_{h} u_{j}^{n})+(I_{h} u_{j}%
^{n}-u_{j,h}^{n})=\eta_{j}^{n}+\xi_{j,h}^{n} \text{
,}
\end{equation}
where $I_{h} u_{j}^{n} \in V_{h} $ is an interpolant of $u_{j}^{n}$
in $V_{h}.$
Subtracting (\ref{error-V}) from (\ref{eq:convtrue}) gives

\begin{equation}
\label{eq:err}
\begin{split}
&\left(\frac{3\xi_{j,h}^{n+1}-4\xi_{j,h}^{n}+\xi_{j,h}^{n-1}}{ 2\Delta t},v_{h}\right) 
+b^{*}\left(u_{j}^{n+1},u_{j}^{n+1},v_{h}\right)+\overline{\nu}\left(\nabla\xi
_{j,h}^{n+1},\nabla v_{h}\right)\\
&\,\,+(\nu_j-\overline{\nu})\left(\nabla(2\xi_{j,h}^n-\xi_{j,h}^{n-1}),\nabla v_{h}\right)-b^{*}\left(2u_{j,h}^{n}-u_{j,h}^{n-1}-u_{j,h}^{\prime n},u_{j,h}^{n+1},v_{h}\right)\\
&\,\ -b^{*}\left(u_{j,h}^{\prime n},2u_{j,h}
^{n}-u_{j,h}^{n-1},v_{h}\right)-\left(p_{j}^{n+1},\nabla\cdot v_{h}\right)\\
&=-\left(\frac{3\eta_{j}^{n+1}-4\eta_{j}^{n}+\eta_j^{n-1}}{2\Delta t},v_{h}\right) -\overline{\nu}\left(\nabla\eta
_{j}^{n+1},\nabla v_{h}\right)+\mathrm{Intp}\left(u_j^{n+1};v_{h}\right)\\
&\,\,+(\overline{\nu}-\nu_j)\left(\nabla(2\eta_{j}^n-\eta_{j}^{n-1}),\nabla v_{h}\right)+(\overline{\nu}-\nu_j)\left(\nabla(u
_{j}^{n+1}-2u_{j}^n+u_{j}^{n-1}),\nabla v_{h}\right)
 \text{ .}\nonumber
\end{split}
\end{equation}

Setting $v_{h}=\xi_{j,h}^{n+1}\in V_{h}$ and rearranging the nonlinear
terms leads to
\begin{equation}
\label{eq:err1}
\begin{split}
&\frac{1}{4\Delta t}\left(\Vert \xi_{j,h}^{n+1}\Vert^{2}+\Vert 2\xi
_{j,h}^{n+1}-\xi_{j,h}^{n}\Vert^{2}\right)-\frac{1}{4\Delta t}\left(\Vert \xi_{j,h}^{n}\Vert^{2}+\Vert 2\xi
_{j,h}^{n}-\xi_{j,h}^{n-1}\Vert^{2}\right)\\
&\,\,+\frac{1}{4\Delta t}\|\xi_{j,h}^{n+1}-2\xi_{j,h}^{n}+\xi_{j,h}^{n-1}\|^{2} +\overline{\nu} \Vert \nabla \xi_{j,h}^{n+1}\Vert^2\\
&=-b^{*}\left(u_{j}^{n+1},u_{j}^{n+1},\xi_{j,h}^{n+1}\right)
+b^{*}\left(2u_{j,h}^{n}-u_{j,h}^{n-1},u_{j,h}^{n+1},\xi_{j,h}^{n+1}\right)\\
&\,\, +b^{*}\left(u_{j,h}^{\prime n},2u_{j,h}
^{n}-u_{j,h}^{n-1}-u_{j,h}^{n+1},\xi_{j,h}^{n+1}\right)+\left(p_{j}^{n+1},\nabla\cdot\xi_{j,h}^{n+1}\right)\\
&\,\,-\left(\frac{3\eta_{j}^{n+1}-4\eta_{j}^{n}+\eta_j^{n-1}}{2\Delta t},\xi_{j,h}^{n+1}\right) -\overline{\nu}\left(\nabla\eta
_{j}^{n+1},\nabla \xi_{j,h}^{n+1}\right)+\mathrm{Intp}\left(u_j^{n+1};\xi_{j,h}^{n+1}\right)\\
&\,\,+(\overline{\nu}-\nu_j)\left(\nabla(2\xi_{j,h}^n-\xi_{j,h}^{n-1}),\nabla \xi_{j,h}^{n+1}\right)+(\overline{\nu}-\nu_j)\left(\nabla(2\eta_{j}^n-\eta_{j}^{n-1}),\nabla \xi_{j,h}^{n+1}\right)\\
&\,\,+(\overline{\nu}-\nu_j)\left(\nabla(u_{j}^{n+1}-2u_{j}^n+u_{j}^{n-1}),\nabla \xi_{j,h}^{n+1}\right)
 \text{ .}
\end{split}
\end{equation}

We first bound the viscous terms on the RHS of \eqref{eq:err1}:

\begin{equation}
\label{firstineq}
\begin{split}
-&(\nu_j-\overline{\nu}) (\nabla (u_{j}^{n+1}-2u_j^n+u_j^{n-1}), \nabla\xi_{j,h}^{n+1})\\
%&\leq \vert \nu_j-\overline{\nu}\vert \|\nabla (u_j^{n+1}-2u_j^n+u_j^{n-1})\| \|\nabla\xi_{j,h}^{n+1}\| \\
&\leq  \frac{1}{4C_0}\frac{\vert \nu_j-\overline{\nu}\vert^2}{\overline{\nu}}\|\nabla(u_{j}^{n+1}-2u_j^n+u_j^{n-1})\|^{2} 
+ C_0\overline{\nu}\|\nabla\xi_{j,h}^{n+1}\|^{2}  \\
&\,\,\leq \frac{\Delta t^3}{4C_0}\frac{\vert \nu_j-\overline{\nu}\vert^2}{\overline{\nu}} \left(\int_{t^{n-1}}^{t^{n+1}}\| \nabla u_{j,tt}\|^{2}\, dt\right)+ C_0\overline{\nu}\|\nabla\xi_{j,h}^{n+1}\|^{2}\,,
\end{split}
\end{equation}
and
\begin{align}
-\overline{\nu}(\nabla\eta_{j}^{n+1},\nabla\xi_{j,h}^{n+1}) 
%&\leq\overline{\nu}\|\nabla\eta_{j}^{n+1}\| \|\nabla\xi_{j,h}^{n+1}\| 
\leq \frac{\overline{\nu}}{4C_0}\|\nabla\eta_{j}^{n+1}\|^{2}+ C_0 \overline{\nu}\|\nabla\xi_{j,h}%
^{n+1}\|^{2} \text{ ,}
\end{align}
\begin{align}
-2(\nu_j-\overline{\nu})(\nabla\eta_{j}^{n},\nabla\xi_{j,h}^{n+1}) 
%&\leq 2\vert \nu_j-\overline{\nu}\vert \|\nabla\eta_{j}^{n}\| \|\nabla\xi_{j,h}^{n+1}\| 
\leq \frac{1}{C_0}\frac{\vert \nu_j-\overline{\nu}\vert^2}{\overline{\nu}}\|\nabla\eta_{j}^{n}\|^{2}
+ C_0\overline{\nu}\|\nabla\xi_{j,h}^{n+1}\|^{2} \text{ ,}
\end{align}

\begin{align}
(\nu_j-\overline{\nu})(\nabla\eta_{j}^{n-1},\nabla\xi_{j,h}^{n+1}) 
%&\leq\vert \nu_j-\overline{\nu}\vert \|\nabla\eta_{j}^{n-1}\| \|\nabla\xi_{j,h}^{n+1}\|
\leq \frac{1}{4C_0}\frac{\vert \nu_j-\overline{\nu}\vert^2}{\overline{\nu}}\|\nabla\eta_{j}^{n-1}\|^{2}
+ C_0\overline{\nu}\|\nabla\xi_{j,h}^{n+1}\|^{2} \text{ ,}
\end{align}
\begin{equation}
\label{ineq:err2}
\begin{split}
2(\nu_j-\overline{\nu})(\nabla\xi_{j,h}^{n},\nabla\xi_{j,h}^{n+1}) 
%&\leq 2\vert \nu_j-\overline{\nu}\vert \|\nabla\xi_{j,h}^{n}\| \|\nabla\xi_{j,h}^{n+1}\| \\
&\leq \frac{1}{C_1}\frac{\vert \nu_j-\overline{\nu}\vert^2}{\overline{\nu}} \|\nabla\xi_{j,h}^{n}\|^{2}+ C_1\overline{\nu}\|\nabla\xi_{j,h}%
^{n+1}\|^{2}  \\
&\leq 
\vert \nu_j-\overline{\nu}\vert \|\nabla\xi_{j,h}^{n}\|^{2}+ \vert \nu_j-\overline{\nu}\vert \|\nabla\xi_{j,h}^{n+1}\|^{2}  ,
\end{split}
\end{equation}
\begin{equation}
\label{ineq:err22}
\begin{split}
-(\nu_j-\overline{\nu})(\nabla\xi_{j,h}^{n-1},\nabla\xi_{j,h}^{n+1}) 
%&\leq\vert \nu_j-\overline{\nu}\vert \|\nabla\xi_{j,h}^{n-1}\| \|\nabla\xi_{j,h}^{n+1}\| \\
&\leq \frac{1}{4C_2}\frac{\vert \nu_j-\overline{\nu}\vert^2}{\overline{\nu}} \|\nabla\xi_{j,h}^{n-1}\|^{2}+ C_2\overline{\nu}\|\nabla\xi_{j,h}%
^{n+1}\|^{2}  \\
&\leq 
\frac{\vert \nu_j-\overline{\nu}\vert}{2} \|\nabla\xi_{j,h}^{n-1}\|^{2}+ \frac{\vert \nu_j-\overline{\nu}\vert}{2} \|\nabla\xi_{j,h}^{n+1}\|^{2}  ,
\end{split}
\end{equation}
where, because the terms on the RHS of \eqref{ineq:err2}  and \eqref{ineq:err22} need to be hidden in the LHS of the error equation, we took $C_1= \frac{|\nu_j-\overline{\nu}|}{\nu}$ and $C_2= \frac{|\nu_j-\overline{\nu}|}{2\nu}$ in order to minimize their summations.

Next, we analyze the nonlinear terms on the RHS of \eqref{eq:err1} one by one. The first two nonlinear terms can be rewritten as
\begin{equation}
\label{eq:nonlinearA}
\begin{split}
&-b^{*}\left(u_{j}^{n+1},u_{j}^{n+1},\xi_{j,h}^{n+1}\right)+b^{*}\left(2u_{j,h}^{n}-u_{j,h}^{n-1}
,u_{j,h}^{n+1},\xi_{j,h}^{n+1}\right)\\
&=-b^{*}\left(2e_j^n-e_j^{n-1},u_{j}^{n+1},\xi_{j,h}^{n+1}\right)-b^{*}\left(2u_{j,h}^{n}-u_{j,h}^{n-1},e^{n+1}
,\xi_{j,h}^{n+1}\right)\\
&\quad-b^{*}\left(u_{j}^{n+1}-2u_j^n+u_j^{n-1},u_{j}^{n+1},\xi_{j,h}^{n+1}\right)\\
&=-b^*\left(2\eta^n_j-\eta^{n-1}_j, u_{j}^{n+1}, \xi_{j,h}^{n+1}\right)-b^*\left(2\xi^n_{j,h}-\xi^{n-1}_{j,h}, u_{j}^{n+1}, \xi_{j,h}^{n+1}\right)\\
&\quad-b^{*}\left(2u_{j,h}^{n}-u_{j,h}^{n-1},\eta_{j}^{n+1},\xi_{j,h}^{n+1}\right)-b^{*}\left(u_{j}^{n+1}
-(2u_{j}^{n}-u_j^{n-1}),u_{j}^{n+1},\xi_{j,h}^{n+1}\right).
\end{split}
\end{equation}
and 
\begin{align*}
-b^*\left(2\eta^n_j-\eta^{n-1}_j, u_{j}^{n+1}, \xi_{j,h}^{n+1}\right) &\leq C\|\nabla
\left(2\eta^n_j-\eta^{n-1}_j\right) \|\|\nabla u_{j}^{n+1}\|\|\nabla\xi_{j,h}
^{n+1}\|\\
&\leq C_0 \overline{\nu}\|\nabla\xi_{j,h}^{n+1}\|^{2}+\frac{C^2}{4C_0\overline{\nu}}\left(\|\nabla\eta^n_j \|^{2}+\Vert \nabla\eta^{n-1}_j \|^{2}\right)\|\nabla u_{j}^{n+1}\|^{2} .
\end{align*}

Since $u_j\in L^{\infty}\left(0,T;H^{1}(\Omega)\right)$, we have the estimates
\begin{equation}
\begin{split}
-2b^*\left(\xi^n_{j,h}, u_{j}^{n+1}, \xi_{j,h}^{n+1}\right) 
&\leq C\|\nabla \xi^n_{j,h} \|^{\frac{1}{2}} \Vert \xi^n_{j,h}\Vert^{\frac{1}{2}}\|\nabla u_{j}^{n+1}\|\|\nabla\xi_{j,h}%
^{n+1}\|\\
&\leq C\|\nabla \xi^n_{j,h} \|^{\frac{1}{2}} \Vert \xi^n_{j,h}\Vert^{\frac{1}{2}}\|\nabla\xi_{j,h}%
^{n+1}\|\\
&\leq C\left(\epsilon \|\nabla \xi_{j,h}^{n+1}\|^{2}+\frac{1}{\epsilon}\|\nabla \xi^n_{j,h} \|\|\xi^n_{j,h} \|\right) \\
&\leq C\left(\epsilon \|\nabla \xi_{j,h}^{n+1}\|^{2}+\frac{1}{\epsilon}\left(\delta \|\nabla \xi^n_{j,h} \|^2+\frac{1}{\delta}\|\xi^n_{j,h} \|^2\right)\right) \\
&\leq C_0\overline{\nu} \|\nabla \xi_{j,h}^{n+1}\|^{2}+C_0\overline{\nu}\|\nabla \xi^n_{j,h} \|^2+CC_0^{-3}\overline{\nu}^{-3}\|\xi^n_{j,h} \|^2 .
\end{split}
\end{equation}
Similarly,
\begin{equation}
\begin{split}
b^*\left(\xi^{n-1}_{j,h}, u_{j}^{n+1}, \xi_{j,h}^{n+1}\right) 
&\leq C\|\nabla \xi^{n-1}_{j,h} \|^{\frac{1}{2}} \Vert \xi^{n-1}_{j,h}\Vert^{\frac{1}{2}}\|\nabla u_{j}^{n+1}\|\|\nabla\xi_{j,h}
^{n+1}\|\\
&\leq C\|\nabla \xi^{n-1}_{j,h} \|^{\frac{1}{2}} \Vert \xi^{n-1}_{j,h}\Vert^{\frac{1}{2}}\|\nabla\xi_{j,h}%
^{n+1}\|\\
&\leq C\left(\epsilon \|\nabla \xi_{j,h}^{n+1}\|^{2}+\frac{1}{\epsilon}\|\nabla \xi^{n-1}_{j,h} \|\|\xi^{n-1}_{j,h} \|\right)\\
&\leq C\left(\epsilon \|\nabla \xi_{j,h}^{n+1}\|^{2}+\frac{1}{\epsilon}\left(\delta \|\nabla \xi^{n-1}_{j,h} \|^2+\frac{1}{\delta}\|\xi^{n-1}_{j,h} \|^2\right)\right) \\
&\leq C_0\overline{\nu} \|\nabla \xi_{j,h}^{n+1}\|^{2}+C_0\overline{\nu}\|\nabla \xi^{n-1}_{j,h} \|^2+CC_0^{-3}\overline{\nu}^{-3}\|\xi^{n-1}_{j,h} \|^2 .
\end{split}
\end{equation}

Also by inequality \eqref{In1} and the stability result \eqref{stability_result}, i.e. $\Vert u_{j,h}^n\Vert^2\leq C$, we have
\begin{equation}
\begin{split}
-2b^{*}\left(u_{j,h}^{n},\eta_{j}^{n+1},\xi_{j,h}^{n+1}\right) 
&\leq C\|\nabla u_{j,h}^{n}\|^{\frac{1}{2}}\|u_{j,h}
^{n}\|^{\frac{1}{2}}\|\nabla\eta_{j}^{n+1}\|\|\nabla\xi_{j,h}^{n+1}\|\\
&\leq C_0 \overline{\nu}\|\nabla\xi_{j,h}^{n+1}\|^{2}+\frac{C^2}{4C_0\overline{\nu}}\|\nabla u_{j,h}
^{n}\|\|\nabla\eta_{j}^{n+1}\|^{2}\text{ .}
\end{split}
\end{equation}

\begin{equation}
\begin{split}
b^{*}\left(u_{j,h}^{n-1},\eta_{j}^{n+1},\xi_{j,h}^{n+1}\right) 
&\leq C\|\nabla u_{j,h}^{n-1}\|^{\frac{1}{2}}\| u_{j,h}
^{n-1}\|^{\frac{1}{2}}\|\nabla\eta_{j}^{n+1}\|\|\nabla\xi_{j,h}^{n+1}\|\\
&\leq C_0 \overline{\nu}\|\nabla\xi_{j,h}^{n+1}\|^{2}+\frac{C^2}{4C_0\overline{\nu}}\|\nabla u_{j,h}
^{n-1}\|\|\nabla\eta_{j}^{n+1}\|^{2}\text{ .}
\end{split}
\end{equation}

\begin{equation}
\begin{split}
&-b^{*}\left(u_{j}^{n+1}-\left(2u_{j}^{n}-u_j^{n-1}\right),u_{j}^{n+1},\xi_{j,h}^{n+1}\right) \\
&\,\,\leq C\|\nabla \left(u_{j}^{n+1}-2u_{j}^{n}+u_j^{n-1}\right)\|\|\nabla u_{j}^{n+1}\|\|\nabla\xi_{j,h}%
^{n+1}\|\\
&\,\,\leq C_0 \overline{\nu}\|\nabla\xi_{j,h}^{n+1}\|^{2}+\frac{C^2}{4C_0\overline{\nu}}\|\nabla \left(u_{j}
^{n+1}-2u_{j}^{n}+u_j^{n-1}\right)\|^{2}\|\nabla u_{j,h}^{n+1}\|^{2}\\
&\,\,\leq C_0\overline{\nu}\|\nabla\xi_{j,h}^{n+1}\|^{2}+\frac{C^2}{4C_0\overline{\nu}}\Delta t^3
\left(\int_{t^{n-1}}^{t^{n+1}}\| \nabla u_{j,tt} \|^{2} dt\right)\|\nabla u_{j}
^{n+1}\|^{2}\text{ . }
\end{split}
\end{equation}

Now we bound the third nonlinear term in \eqref{eq:err1}:
\begin{equation}
\label{eq:nonlinearB}
\begin{split}
&-b^{*}\left(u_{j,h}^{\prime n},u^{n+1}_{j,h}-2u_{j,h}^{n}+u_{j,h}^{n-1},\xi_{j,h}^{n+1}
\right)\\
&=b^{*}\left(u_{j,h}^{\prime n},e_j^{n+1}-2e_j^n+e_j^{n-1},\xi_{j,h}^{n+1}
\right)-b^{*}\left(u_{j,h}^{\prime n},u_j^{n+1}-2u_j^n+u_j^{n-1},\xi_{j,h}^{n+1}
\right)\\
&=b^{*}\left(u_{j,h}^{\prime n},\xi_{j,h}^{n+1}-2\xi_{j,h}^{n}+\xi_{j,h}^{n-1},\xi_{j,h}^{n+1}
\right)+b^{*}\left(u_{j,h}^{\prime n},\eta_j^{n+1}-2\eta_j^{n}+\eta_j^{n-1},\xi_{j,h}^{n+1}
\right)\\
&\,\,-b^{*}\left(u_{j,h}^{\prime n},u^{n+1}_{j}-2u_{j}^{n}+u_{j}^{n-1},\xi_{j,h}^{n+1}
\right)\text{ .}
\end{split}
\end{equation}

By skew symmetry
\begin{gather}
-b^{*}\left(u_{j,h}^{\prime n},\xi_{j,h}^{n+1}-2\xi_{j,h}^{n}+\xi_{j,h}^{n-1},\xi_{j,h}^{n+1}\right) 
=b^{*}\left(u_{j,h}^{\prime n},\xi_{j,h}^{n+1}, \xi_{j,h}^{n+1}-2\xi_{j,h}^{n}+\xi_{j,h}^{n-1}\right) .\nonumber
\end{gather}
Using \eqref{In2} and inverse inequality \eqref{inverse} gives
\begin{equation}
\begin{split}
&\qquad b^{*}\left(u_{j,h}^{\prime n},2\xi_{j,h}^{n}-\xi_{j,h}^{n-1}-\xi_{j,h}^{n+1},\xi_{j,h}^{n+1}%
\right) \\
&\leq C \|\nabla 
 u_{j,h}^{\prime n} \| \|\nabla \xi_{j,h}^{n+1}\|\| \nabla (\xi_{j,h}
^{n+1}-2\xi_{j,h}^{n}+\xi_{j,h}^{n-1})\|^{1/2}\| \xi_{j,h}%
^{n+1}-2\xi_{j,h}^{n}+\xi_{j,h}^{n-1}\|^{1/2}\\
&\leq C \|\nabla 
 u_{j,h}^{\prime n} \| \|\nabla \xi_{j,h}^{n+1}\|\left(h^{-1/2}\right)\| \xi_{j,h}
^{n+1}-2\xi_{j,h}^{n}+\xi_{j,h}^{n-1}\|\\
&\leq\frac{1}{8\Delta t}\|\xi_{j,h}
^{n+1}-2\xi_{j,h}^{n}+\xi_{j,h}^{n-1}\|^{2}+C\frac{\Delta t}{h}\|
 \nabla 
 u_{j,h}^{\prime n} \|^2\|\nabla \xi^{n+1}_{j,h} \|^{2}.
\end{split}
\end{equation}

\begin{equation}
\begin{split}
&\qquad b^{*}\left(u_{j,h}^{\prime n},\eta_{j}^{n+1}-2\eta_{j}^{n}+\eta_j^{n-1},\xi_{j,h}^{n+1}\right)\\
&\leq C\|\nabla
u_{j,h}^{\prime n}\|\|\nabla\left(\eta_{j}^{n+1}-2\eta_{j}^{n}+\eta_j^{n-1}\right)\|\|\nabla\xi_{j,h}%
^{n+1}\|\\
&\leq C_0\overline{\nu}\|\nabla\xi_{j,h}^{n+1}\|^{2}+CC_0^{-1}\overline{\nu}^{-1}\|\nabla u_{j,h}^{\prime n}\|^{2}\|\nabla\left(\eta_{j}^{n+1}-2\eta_{j}^{n}+\eta_j^{n-1}\right)\|^{2}\\
&\leq C_0\overline{\nu}\|\nabla\xi_{j,h}^{n+1}\|^{2}+\frac{C \Delta t^3}{C_0\overline{\nu}}\|\nabla
u_{j,h}^{\prime n}\|^{2}\left(\int_{t^{n-1}}^{t^{n+1}}\| \nabla \eta_{j,tt}\|^{2} \text{ }dt\right)\text{
.}
\end{split}
\end{equation}

\begin{equation}
\begin{split}
&\qquad b^{*}\left(u_{j,h}^{\prime n},u_{j}^{n+1}-2u_{j}^{n}+u_j^{n-1},\xi_{j,h}^{n+1}\right)\\
&\leq C\|\nabla u_{j,h}^{\prime n}\|\|\nabla\left(u_{j}^{n+1}-2u_{j}^{n}+u_j^{n-1}\right)\|\|\nabla\xi_{j,h}%
^{n+1}\|\\
&\leq C_0\overline{\nu}\|\nabla\xi_{j,h}^{n+1}\|^{2}+CC_0^{-1}\overline{\nu}^{-1}\|\nabla u_{j,h}^{\prime n}\|^{2}\|\nabla\left(u_{j}^{n+1}-2u_{j}^{n}+u_j^{n-1}\right)\|^{2}\\
&\leq C_0\overline{\nu}\|\nabla\xi_{j,h}^{n+1}\|^{2}+C C_0^{-1}\overline{\nu}^{-1} \Delta t^3\|\nabla
u_{j,h}^{\prime n}\|^{2}\left(\int_{t^{n-1}}^{t^{n+1}}\| \nabla u_{j,tt}\|^{2} \text{ }dt\right)\text{
.}
\end{split}
\end{equation}
For the pressure term in \eqref{eq:err1}, since $\xi_{j,h}^{n+1}\in V_{h}$, we have%
\begin{align}
(p_{j}^{n+1},\nabla\cdot\xi_{j,h}^{n+1})&=(p_{j}^{n+1}-q_{j,h}^{n+1},
\nabla\cdot\xi_{j,h}^{n+1})\nonumber\\
&\leq\sqrt{d}\,\|p_{j}^{n+1}-q_{j,h}^{n+1}\|\|\nabla\xi_{j,h}^{n+1}\|\\
&\leq
\frac{d}{4\, C_0} \overline{\nu}^{-1}\|p_{j}^{n+1}-q_{j,h}^{n+1}\|^{2} + C_0\,\overline{\nu}\|\nabla\xi_{j,h}^{n+1}\|^{2}
 \text{ .}\nonumber
\end{align}
The other terms are bounded as
\begin{equation}
\begin{split}
\left(\frac{3\eta_{j}^{n+1}-4\eta_{j}^{n}+\eta_j^{n-1}}{ 2\Delta t},\xi_{j,h}^{n+1}\right)
%\leq C\|\frac{3\eta_{j}^{n+1}-4\eta_{j}^{n}+\eta_j^{n-1}}{ 2\Delta t}\| \|\nabla\xi_{j,h}^{n+1}\|\\
&\leq \frac{C}{4C_0} \overline{\nu}^{-1}\|\frac{3\eta_{j}^{n+1}-4\eta_{j}^{n}+\eta_j^{n-1}}{ 2\Delta t}\|^{2}
+C_0\overline{\nu}\|\nabla\xi_{j,h}^{n+1}\|^{2}\nonumber\\
&\leq \frac{C}{4C_0} \overline{\nu}^{-1}\|\frac{1}{\Delta t}\int_{t^{n-1}}^{t^{n+1}} \eta_{j,t} \text{ } dt
\|^2+C_0 \overline{\nu}\|\nabla\xi_{j,h}^{n+1}\|^{2}\nonumber\\
&\leq\frac{C}{4C_0\overline{\nu}\Delta t}\int_{t^{n-1}}^{t^{n+1}}\| \eta_{j,t}\|^{2}\text{ }
dt+C_0\overline{\nu}\|\nabla\xi_{j,h}^{n+1}\|^{2}\text{ .}\nonumber
\end{split}
\end{equation}
and
\begin{equation}
\label{lastineq}
\begin{split}
\mathrm{Intp}\left(u_{j}^{n+1};\xi_{j,h}^{n+1}\right)&=\left(\frac{3u_{j}^{n+1}-4u_{j}^{n}+u_j^{n-1}}{2\Delta
t}-u_{j,t}(t^{n+1}),\xi_{j,h}^{n+1}\right)\\
&\leq C\|\frac{3u_{j}^{n+1}-4u_{j}^{n}+u_j^{n-1}}{2\Delta t}-u_{j,t}(t^{n+1})\|
\|\nabla\xi_{j,h}^{n+1}\|\\
&\leq C_0\overline{\nu}\|\nabla\xi_{j,h}^{n+1}\|^{2}+\frac{C^2}{4C_0\overline{\nu}}\|\frac
{3u_{j}^{n+1}-4u_{j}^{n}+u_j^{n-1}}{2\Delta t}-u_{j,t}(t^{n+1})\|^{2}\\
&\leq C_0 \overline{\nu}\|\nabla\xi_{j,h}^{n+1}\|^{2}+\frac{5C^2\Delta t^3}{8C_0\overline{\nu}}
\int_{t^{n-1}}^{t^{n+1}}\|u_{j,ttt}\|^{2} dt \text{ .}
\end{split}
\end{equation}

Combining \eqref{firstineq}-\eqref{lastineq} and taking $C_0= \frac{1}{17} \frac{\epsilon} {\sqrt{\mu}+\epsilon}( 1-\frac{\sqrt{\mu}}{2}) $, we have for $\forall \sigma$, $0<\sigma<1$,
{\allowdisplaybreaks
\begin{align}
&\frac{1}{4\Delta t}\left(\Vert \xi_{j,h}^{n+1}\Vert^{2}+\Vert 2\xi
_{j,h}^{n+1}-\xi_{j,h}^{n}\Vert^{2}\right)-\frac{1}{4\Delta t}\left(\Vert \xi_{j,h}^{n}\Vert^{2}+\Vert 2\xi
_{j,h}^{n}-\xi_{j,h}^{n-1}\Vert^{2}\right) \nonumber\\
&\quad+\frac{1}{8\Delta t}\|\xi_{j,h}^{n+1}-2\xi_{j,h}^{n}+\xi_{j,h}^{n-1}\|^{2}\nonumber\\
&\quad+2C_0\overline{\nu}\left(\Vert\nabla\xi_{j,h}^{n+1}\Vert^{2}-\Vert\nabla\xi_{j,h}^{n}\Vert^{2}\right)
+C_0\overline{\nu}\left(\Vert\nabla\xi_{j,h}^{n}\Vert^{2}-\Vert\nabla\xi_{j,h}^{n-1}\Vert^{2}\right)\nonumber\\
&\quad+\overline{\nu}\Delta t \sigma \left(1-17C_0-\frac{3\vert \nu_j-\overline{\nu}\vert}{2\overline{\nu}}\right)
\left(\Vert\nabla\xi_{j,h}^{n+1}\Vert^{2}-\Vert\nabla\xi_{j,h}^{n}\Vert^{2}\right)\nonumber\\
&\quad+\overline{\nu}\Delta t \left( (1-\sigma)\left(1-17C_0-\frac{3\vert \nu_j-\overline{\nu}\vert}{2\overline{\nu}}\right)-\frac{C\Delta t}{\overline{\nu} h}\| \nabla 
 u_{j,h}^{\prime n} \|^2\right)\|\nabla \xi^{n+1}_{j,h} \|^{2}\nonumber\\
&\quad+\overline{\nu}\Delta t \left(\frac{2}{3}\sigma \left(1-17C_0-\frac{3\vert \nu_j-\overline{\nu}\vert}{2\overline{\nu}}\right)-\frac{\vert \nu_j-\overline{\nu}\vert}{\overline{\nu}}\right)\Vert\nabla\xi_{j,h}^{n}\Vert^{2}\nonumber\\
&\quad+\overline{\nu}\Delta t \left(\frac{1}{3}\sigma \left(1-17C_0-\frac{3\vert \nu_j-\overline{\nu}\vert}{2\overline{\nu}}\right)-\frac{\vert \nu_j-\overline{\nu}\vert}{2\overline{\nu}}\right)\left(\Vert\nabla\xi_{j,h}^{n}\Vert^{2}-\Vert\nabla\xi_{j,h}^{n-1}\Vert^{2}\right)\nonumber\\
&\leq CC_0^{-3}\nu^{-3}\left(\|\xi^n_{j,h} \|^2+\|\xi^{n-1}_{j,h} \|^2\right)\label{eq:err2}\\
&\quad+\frac{C}{4C_0\overline{\nu}}\|\nabla u_{j,h}^{n}\|\|\nabla\eta_{j}^{n+1}\|^{2}+\frac{C}{4C_0\overline{\nu}}\|\nabla u_{j,h}^{n-1}\|\|\nabla\eta_{j}^{n+1}\|^{2} \nonumber\\
&\quad+\frac{\Delta t^3}{4C_0}\frac{\vert \nu_j-\overline{\nu}\vert^2}{\overline{\nu}}\left(\int_{t^{n-1}}^{t^{n+1}}\| \nabla u_{j,tt} \|^{2}
dt\right)+\frac{C\Delta t^3}{4C_0\overline{\nu}}\left(\int_{t^{n-1}}^{t^{n+1}}\| \nabla u_{j,tt} \|^{2}
dt\right)\Vert \nabla u_{j}^{n+1}\Vert^2\nonumber\\
&\quad+\frac{C}{4C_0\overline{\nu}}\left(\|\nabla\eta^n_j \|^{2}+\Vert \nabla\eta^{n-1}_j \|^{2}\right)\|\nabla u_{j}^{n+1}\|^{2}
+\frac{C \Delta t^3}{C_0\overline{\nu}}\|\nabla
u_{j,h}^{\prime n}\|^{2}\left(\int_{t^{n-1}}^{t^{n+1}}\| \nabla u_{j,tt}\|^{2} \text{ }dt\right)\nonumber\\
&\quad+\frac{C \Delta t^3}{C_0\overline{\nu}}\|\nabla
u_{j,h}^{\prime n}\|^{2}\left(\int_{t^{n-1}}^{t^{n+1}}\| \nabla \eta_{j,tt}\|^{2} \text{ }dt\right)
+\frac{d}{4C_0\overline{\nu}}\|p_{j}^{n+1}-q_{j,h}^{n+1}\|^{2}\nonumber\\
&\quad+\frac{C}{4C_0\overline{\nu}\Delta t}\int_{t^{n-1}}^{t^{n+1}}\| \eta_{j,t}\|^{2}\text{ }dt
+\frac{\overline{\nu}}{4C_0}\|\nabla\eta_{j}^{n+1}\|^{2}+\frac{1}{C_0}\frac{\vert \nu_j-\overline{\nu}\vert^2}{\overline{\nu}}\|\nabla\eta_{j}^{n}\|^{2}\nonumber\\
&\quad+\frac{1}{4C_0}\frac{\vert \nu_j-\overline{\nu}\vert^2}{\overline{\nu}}\|\nabla\eta_{j}^{n-1}\|^{2}
+\frac{C\Delta t^3}{4C_0\overline{\nu}}
\int_{t^{n-1}}^{t^{n+1}}\|u_{j,ttt}\|^{2} dt, 
\end{align}
}

\noindent where $C$ on the RHS is a generic constant independent of $\Delta t$ and $h$.

Similar to the discussion in the stability proof, we take $\sigma= \frac{\sqrt{\mu}+\epsilon}{2-\sqrt{\mu}}$,  where $\epsilon \in (0, 2-2\sqrt{\mu})$. Then by the viscosity deviation condition \eqref{conv1}, we have 
\begin{align}
1-17C_0-\frac{3\vert \nu_j-\overline{\nu}\vert}{2\overline{\nu}}
&=\frac{(2+\epsilon)\sqrt{\mu}}{2(\sqrt{\mu}+\epsilon)}-\frac{3\vert \nu_j-\overline{\nu}\vert}{2\overline{\nu}}\\
&>\frac{(2+\epsilon)\sqrt{\mu}}{2(\sqrt{\mu}+\epsilon)}-\frac{\sqrt{\mu}}{2}=\frac{\sqrt{\mu}(2-\sqrt{\mu})}{2(\sqrt{\mu}+\epsilon)}>0,\nonumber
\end{align}
\begin{align}
&\frac{2}{3}\sigma \left(1-17C_0-\frac{3\vert \nu_j-\overline{\nu}\vert}{2\overline{\nu}}\right)-\frac{\vert \nu_j-\overline{\nu}\vert}{\overline{\nu}}>\frac{2}{3}\frac{\sqrt{\mu}+\epsilon}{2-\sqrt{\mu}}\frac{\sqrt{\mu}(2-\sqrt{\mu})}{2(\sqrt{\mu}+\epsilon)}-\frac{\sqrt{\mu}}{3}=0,
\end{align}
\begin{align}
&\text{ and }\qquad\frac{1}{3}\sigma \left(1-17C_0-\frac{3\vert \nu_j-\overline{\nu}\vert}{2\overline{\nu}}\right)-\frac{\vert \nu_j-\overline{\nu}\vert}{2\overline{\nu}}>0.
\end{align}
Also, by the stability condition \eqref{ineq:CFL-h2}, we have
\begin{align}
 &(1-\sigma)\left(1-17C_0-\frac{3\vert \nu_j-\overline{\nu}\vert}{2\overline{\nu}}\right)-\frac{C\Delta t}{\overline{\nu} h}\| \nabla 
 u_{j,h}^{\prime n} \|^2\\
&\qquad\qquad\qquad=\frac{2-2\sqrt{\mu}-\epsilon}{2-\sqrt{\mu}}\frac{\sqrt{\mu}(2-\sqrt{\mu})}{2(\sqrt{\mu}+\epsilon)}-C\frac{\triangle t}{\overline{\nu} h}\Vert\nabla u_{j,h}^{\prime n}\Vert^{2} \nonumber\\
&\qquad\qquad\qquad>\frac{(2-2\sqrt{\mu}-\epsilon)\sqrt{\mu}}{2(\sqrt{\mu}+\epsilon)}-\frac{(2-2\sqrt{\mu}-\epsilon)\sqrt{\mu}}{2(\sqrt{\mu}+\epsilon)}=0.\nonumber
\end{align}

Then \eqref{eq:err2} reduces to

\begin{align}
&\frac{1}{4\Delta t}\left(\Vert \xi_{j,h}^{n+1}\Vert^{2}+\Vert 2\xi
_{j,h}^{n+1}-\xi_{j,h}^{n}\Vert^{2}\right)-\frac{1}{4\Delta t}\left(\Vert \xi_{j,h}^{n}\Vert^{2}+\Vert 2\xi
_{j,h}^{n}-\xi_{j,h}^{n-1}\Vert^{2}\right)\nonumber \\
&\quad+\frac{1}{8\Delta t}\|\xi_{j,h}^{n+1}-2\xi_{j,h}^{n}+\xi_{j,h}^{n-1}\|^{2}\nonumber\\
&\quad+2C_0\overline{\nu}\left(\Vert\nabla\xi_{j,h}^{n+1}\Vert^{2}-\Vert\nabla\xi_{j,h}^{n}\Vert^{2}\right)
+C_0\overline{\nu}\left(\Vert\nabla\xi_{j,h}^{n}\Vert^{2}-\Vert\nabla\xi_{j,h}^{n-1}\Vert^{2}\right)\nonumber\\
&\leq CC_0^{-3}\nu^{-3}\left(\|\xi^n_{j,h} \|^2+\|\xi^{n-1}_{j,h} \|^2\right)\label{eq:err2b}\\
&\quad +\frac{C}{4C_0\overline{\nu}}\|\nabla u_{j,h}^{n}\|\|\nabla\eta_{j}^{n+1}\|^{2}+\frac{C}{4C_0\overline{\nu}}\|\nabla u_{j,h}^{n-1}\|\|\nabla\eta_{j}^{n+1}\|^{2}\nonumber\\
&\quad+\frac{\Delta t^3}{4C_0}\frac{\vert \nu_j-\overline{\nu}\vert^2}{\overline{\nu}}\left(\int_{t^{n-1}}^{t^{n+1}}\| \nabla u_{j,tt} \|^{2}
dt\right)+\frac{C\Delta t^3}{4C_0\overline{\nu}}\left(\int_{t^{n-1}}^{t^{n+1}}\| \nabla u_{j,tt} \|^{2}
dt\right)\Vert \nabla u_{j}^{n+1}\Vert^2\nonumber\\
&\quad+\frac{C}{4C_0\overline{\nu}}\left(\|\nabla\eta^n_j \|^{2}+\Vert \nabla\eta^{n-1}_j \|^{2}\right)\|\nabla u_{j}^{n+1}\|^{2}
+\frac{C \Delta t^3}{C_0\overline{\nu}}\|\nabla
u_{j,h}^{\prime n}\|^{2}\left(\int_{t^{n-1}}^{t^{n+1}}\| \nabla u_{j,tt}\|^{2} \text{ }dt\right)\nonumber\\
&\quad+\frac{C \Delta t^3}{C_0\overline{\nu}}\|\nabla
u_{j,h}^{\prime n}\|^{2}\left(\int_{t^{n-1}}^{t^{n+1}}\| \nabla \eta_{j,tt}\|^{2} \text{ }dt\right)
+\frac{d}{4C_0\overline{\nu}}\|p_{j}^{n+1}-q_{j,h}^{n+1}%
\|^{2}\nonumber\\
&\quad+\frac{C}{4C_0\overline{\nu}\Delta t}\int_{t^{n-1}}^{t^{n+1}}\| \eta_{j,t}\|^{2}\text{ }
dt
+\frac{\overline{\nu}}{4C_0}\|\nabla\eta_{j}^{n+1}\|^{2}+\frac{1}{C_0}\frac{\vert \nu_j-\overline{\nu}\vert^2}{\overline{\nu}}\|\nabla\eta_{j}^{n}\|^{2}\nonumber\\
&\quad+\frac{1}{4C_0}\frac{\vert \nu_j-\overline{\nu}\vert^2}{\overline{\nu}}\|\nabla\eta_{j}^{n-1}\|^{2}
+\frac{C\Delta t^3}{4C_0\overline{\nu}}
\int_{t^{n-1}}^{t^{n+1}}\|u_{j,ttt}\|^{2} dt .\nonumber
\end{align}

Summing (\ref{eq:err2b}) from $n=1$ to $N-1$ and multiplying both sides by $\Delta t$, and absorbing constants gives
\begin{align*}
&\frac{1}{4}\left(\|\xi_{j,h}^{N}\|^{2}+\Vert 2\xi_{j,h}^{N}-\xi_{j,h}^{N-1}\Vert^2\right)+\sum_{n=1}^{N-1}\frac{1}{8}\Vert \xi_{j,h}^{n+1}-2\xi_{j,h}^n+\xi_{j,h}^{n-1}\Vert^2\\
&\quad+ 2C_0\overline{\nu}\Delta t \|\nabla\xi_{j,h}^{N}\|^{2}+C_0\overline{\nu} \Delta t \Vert\nabla\xi_{j,h}^{N-1}\Vert
^{2}\\
&\leq\frac{1}{4}\left(\|\xi_{j,h}^{1}\|^{2}+\Vert 2\xi_{j,h}^{1}-\xi_{j,h}^{0}\Vert^2\right)+ 2C_0\overline{\nu}\Delta t \|\nabla\xi_{j,h}^{1}\|^{2}+C_0\overline{\nu} \Delta t \Vert\nabla\xi_{j,h}^{0}\Vert
^{2}\\
&\quad+\frac{C\Delta t}{\overline{\nu}^{3}}\sum_{n=0}^{N-1}\Vert\xi_{j,h}^{n}\Vert^{2}
+\Delta t\sum_{n=1}^{N-1}
\bigg\{C \overline{\nu}^{-1}\|\nabla u_{j,h}^{n}\|\|\nabla\eta_{j}^{n+1}\|^{2}\nonumber\\
&\quad+C\overline{\nu}^{-1}\|\nabla u_{j,h}^{n-1}\|\|\nabla\eta_{j}^{n+1}\|^{2}
+C\Delta t^3\frac{\vert \nu_j-\overline{\nu}\vert^2}{\overline{\nu}}\left(\int_{t^{n-1}}^{t^{n+1}}\| \nabla u_{j,tt} \|^{2}
dt\right)\nonumber\\
&\quad+C\Delta t^3\overline{\nu}^{-1}\left(\int_{t^{n-1}}^{t^{n+1}}\| \nabla u_{j,tt} \|^{2}
dt\right)\Vert \nabla u_{j}^{n+1}\Vert^2
+C\overline{\nu}^{-1}\left(\|\nabla\eta^n_j \|^{2}+\Vert \nabla\eta^{n-1}_j \|^{2}\right)\|\nabla u_{j}^{n+1}\|^{2}\nonumber\\
&\quad+C \Delta t^3\overline{\nu}^{-1}\|\nabla
u_{j,h}^{\prime n}\|^{2}\left(\int_{t^{n-1}}^{t^{n+1}}\| \nabla u_{j,tt}\|^{2} \text{ }dt\right)
+C \Delta t^3\overline{\nu}^{-1}\|\nabla
u_{j,h}^{\prime n}\|^{2}\left(\int_{t^{n-1}}^{t^{n+1}}\| \nabla \eta_{j,tt}\|^{2} \text{ }dt\right)\nonumber\\
&\quad+C\overline{\nu}^{-1}\|p_{j}^{n+1}-q_{j,h}^{n+1}%
\|^{2}
+C\overline{\nu}^{-1}\Delta t^{-1}\int_{t^{n-1}}^{t^{n+1}}\| \eta_{j,t}\|^{2}\text{ }
dt+C\overline{\nu}\|\nabla\eta_{j}^{n+1}\|^{2}\nonumber\\
&\quad+C\frac{\vert \nu_j-\overline{\nu}\vert^2}{\overline{\nu}}\|\nabla\eta_{j}^{n}\|^{2}
+C\frac{\vert \nu_j-\overline{\nu}\vert^2}{\overline{\nu}}\|\nabla\eta_{j}^{n-1}\|^{2}
+C\Delta t^3\overline{\nu}^{-1}
\int_{t^{n-1}}^{t^{n+1}}\|u_{j,ttt}\|^{2} dt
\bigg\} \text{ .}\nonumber
\end{align*}
\label{eq:err3}
Using the interpolation inequality \eqref{interp2} and  the result \eqref{eq:sta} from the stability analysis, i.e., $\Delta t \sum_{n=1}^{N-1}\Vert \nabla u_{j,h}^{n+1} \Vert^2 \leq C$, we have
\begin{equation}
\begin{split}
C\overline{\nu}^{-1}\Delta t\sum_{n=1}^{N-1}\Vert\nabla u_{j,h}^{n}\Vert& \Vert\nabla\eta_{j}^{n+1}\Vert^{2}
\leq C\overline{\nu}^{-1}h^{2k}\Delta t \sum_{n=1}^{N-1}\Vert\nabla u_{j,h}^{n}\Vert\Vert u_j^{n+1}\Vert^2_{k+1}\\
\leq & C\nu^{-1}h^{2k}\left( \Delta t \sum_{n=1}^{N-1}\Vert u_{j}^{n+1}\Vert_{k+1}^4+\Delta t \sum_{n=1}^{N-1}\Vert\nabla u_{j,h}^{n}\Vert^2 \right)\\
\leq & C\nu^{-1}h^{2k}\vertiii{ u_j}^4_{4, k+1}+C\nu^{-1}h^{2k}\, ,
\label{ineq:err4}
\end{split}
\end{equation}
and
\begin{equation}
\label{ineq:err4b}
\begin{split}
C\overline{\nu}^{-1}\Delta t\sum_{n=1}^{N-1}\Vert\nabla u_{j,h}^{n-1}\Vert& \Vert\nabla\eta_{j}^{n+1}\Vert^{2}
\leq C\overline{\nu}^{-1}h^{2k}\Delta t \sum_{n=1}^{N-1}\Vert\nabla u_{j,h}^{n-1}\Vert\Vert u_j^{n+1}\Vert^2_{k+1}\\
\leq & C\nu^{-1}h^{2k}\left( \Delta t \sum_{n=1}^{N-1}\Vert u_{j}^{n+1}\Vert_{k+1}^4+\Delta t \sum_{n=1}^{N-1}\Vert\nabla u_{j,h}^{n-1}\Vert^2 \right)\\
\leq & C\nu^{-1}h^{2k}\vertiii{ u_j}^4_{4, k+1}+C\nu^{-1}h^{2k}\, ,
\end{split}
\end{equation}
Because $u_j\in L^{\infty}\left(0,T;H^{1}(\Omega)\right)$, we have $\Vert \nabla u_j^{n+1}\Vert^2\leq C$. Using convergence condition \eqref{conv1} and applying interpolation inequalities \eqref{Interp1}, \eqref{interp2} and \eqref{interp3} gives
\begin{align}\label{ineq:errlast}
&\frac{1}{4}\left(\|\xi_{j,h}^{N}\|^{2}+\Vert 2\xi_{j,h}^{N}-\xi_{j,h}^{N-1}\Vert^2\right)+\sum_{n=1}^{N-1}\frac{1}{8}\Vert \xi_{j,h}^{n+1}-2\xi_{j,h}^n+\xi_{j,h}^{n-1}\Vert^2\\
&\quad+ 2C_0\overline{\nu}\Delta t \|\nabla\xi_{j,h}^{N}\|^{2}+C_0\overline{\nu} \Delta t \Vert\nabla\xi_{j,h}^{N-1}\Vert
^{2}\nonumber\\
&\leq\frac{1}{4}\left(\|\xi_{j,h}^{1}\|^{2}+\Vert 2\xi_{j,h}^{1}-\xi_{j,h}^{0}\Vert^2\right)+ 2C_0\overline{\nu}\Delta t \|\nabla\xi_{j,h}^{1}\|^{2}+C_0\overline{\nu} \Delta t \Vert\nabla\xi_{j,h}^{0}\Vert
^{2}\nonumber\\
&\quad+\frac{C\Delta t}{\overline{\nu}^{3}}\sum_{n=0}^{N-1}\Vert\xi_{j,h}^{n}\Vert^{2}+ C\overline{\nu}^{-1}h^{2k}\vertiii{ u_j}^4_{4, k+1}+C\overline{\nu}^{-1}h^{2k} +C\Delta t^4\frac{\vert \nu_j -\overline{\nu}\vert^2}{\overline{\nu}}\vertiii{ \nabla u_{j,tt}}^2_{2,0}\nonumber\\
&\quad+C\overline{\nu}^{-1}\Delta t^{4}\Vert|u_{j,tt}|\Vert_{2,0}^{2}+ C\overline{\nu}^{-1}h^{2k}\vertiii{ u_j}^2_{2,k+1}+C h \Delta t^3 \vertiii{ \nabla u_{j,tt}}^2_{2,0}\nonumber\\
&\quad+C h^{2k+1} \Delta t^3 \vertiii{ \nabla u_{j,tt}}^2_{2,k+1}+C\overline{\nu}^{-1} h^{2s+2}\Vert|p_{j}|\Vert_{2,s+1}^{2}+C\overline{\nu}^{-1} h^{2k+2}\Vert
|u_{j,t}|\Vert_{2,k+1}^{2}\nonumber\\
&\quad+C\overline{\nu}h^{2k}\vertiii{ u_j }^2_{2,k+1}+C\frac{\vert \nu_j - \overline{\nu}\vert^2}{\overline{\nu}}h^{2k}\vertiii{ u_j}^2_{2,k+1}+C\overline{\nu}^{-1}\Delta t^4\vertiii{ \nabla u_{j,ttt}}^2_{2,0} 
\nonumber
\end{align}
The next step uses an application of the discrete Gronwall
inequality (Girault and Raviart \cite{GR79}, p. 176):
\begin{align}\label{ineq:errlast1}
&\frac{1}{4}\left(\|\xi_{j,h}^{N}\|^{2}+\Vert 2\xi_{j,h}^{N}-\xi_{j,h}^{N-1}\Vert^2\right)+\sum_{n=1}^{N-1}\frac{1}{8}\Vert \xi_{j,h}^{n+1}-2\xi_{j,h}^n+\xi_{j,h}^{n-1}\Vert^2\\
&\quad+ 2C_0\overline{\nu}\Delta t \|\nabla\xi_{j,h}^{N}\|^{2}+C_0\overline{\nu} \Delta t \Vert\nabla\xi_{j,h}^{N-1}\Vert
^{2}\nonumber\\
&\leq e^{\frac{CT}{\nu^{3}}}
\bigg\{\frac{1}{4}\left(\|\xi_{j,h}^{1}\|^{2}+\Vert 2\xi_{j,h}^{1}-\xi_{j,h}^{0}\Vert^2\right)+ 2C_0\overline{\nu}\Delta t \|\nabla\xi_{j,h}^{1}\|^{2}+C_0\overline{\nu} \Delta t \Vert\nabla\xi_{j,h}^{0}\Vert
^{2}\nonumber\\
&\quad+ C\overline{\nu}^{-1}h^{2k}\vertiii{ u_j}^4_{4, k+1}+C\overline{\nu}^{-1}h^{2k} +C\Delta t^4\frac{\vert \nu_j -\overline{\nu}\vert^2}{\overline{\nu}}\vertiii{ \nabla u_{j,tt}}^2_{2,0}\nonumber\\
&\quad+C\overline{\nu}^{-1}\Delta t^{4}\Vert|u_{j,tt}|\Vert_{2,0}^{2}+ C\overline{\nu}^{-1}h^{2k}\vertiii{ u_j}^2_{2,k+1}+C h \Delta t^3 \vertiii{ \nabla u_{j,tt}}^2_{2,0}\nonumber\\
&\quad+C h^{2k+1} \Delta t^3 \vertiii{ \nabla u_{j,tt}}^2_{2,k+1}+C\overline{\nu}^{-1} h^{2s+2}\Vert|p_{j}|\Vert_{2,s+1}^{2}+C\overline{\nu}^{-1} h^{2k+2}\Vert
|u_{j,t}|\Vert_{2,k+1}^{2}\nonumber\\
&\quad+C\overline{\nu}h^{2k}\vertiii{ u_j }^2_{2,k+1}+C\frac{\vert \nu_j - \overline{\nu}\vert^2}{\overline{\nu}}h^{2k}\vertiii{ u_j}^2_{2,k+1}+C\overline{\nu}^{-1}\Delta t^4\vertiii{ \nabla u_{j,ttt}}^2_{2,0} 
\bigg\}\text{ .}\nonumber
\end{align}

Recall that $e_{j}^{n}=\eta_{j}^{n}+\xi_{j,h}^{n}$. Using the triangle
inequality on the error equation to split the error terms into the terms of
$\eta_{j}^{n}$ and $\xi_{j,h}^{n}$ gives
\begin{equation*}
\begin{aligned}
&\frac{1}{4}\|e_j^{N}\|^{2}+ 2C_0\overline{\nu}\Delta t \|\nabla e_{j}^{N}\|^{2}\\
&\qquad\qquad\leq\frac{1}{4}\|\xi_{j,h}^{N}\|^{2}+ 2C_0\overline{\nu}\Delta t \|\nabla\xi_{j,h}^{N}\|^{2}
+\frac{1}{4}\Vert\eta_j^{N}\|^{2}+ 2C_0\overline{\nu}\Delta t \|\nabla\eta_{j}^{N}\|^{2}\text{ ,}
\end{aligned}
\end{equation*}
and 
\begin{equation*}
\begin{aligned}
&\frac{1}{4}\left(\|\xi_{j,h}^{1}\|^{2}+\Vert 2\xi_{j,h}^{1}-\xi_{j,h}^{0}\Vert^2\right)+ 2C_0\overline{\nu}\Delta t \|\nabla\xi_{j,h}^{1}\|^{2}+C_0\overline{\nu} \Delta t \Vert\nabla\xi_{j,h}^{0}\Vert
^{2}\\
&\leq \frac{1}{4}\left(\|e_{j}^{1}\|^{2}+\Vert 2e_{j}^{1}-e_{j}^{0}\Vert^2\right)+ 2C_0\overline{\nu}\Delta t \|\nabla e_{j}^{1}\|^{2}+C_0\overline{\nu} \Delta t \Vert\nabla e_{j}^{0}\Vert
^{2}\\
&\,\,+\frac{1}{4}\left(\|\eta_{j}^{1}\|^{2}+\Vert 2\eta_{j}^{1}-\eta_{j}^{0}\Vert^2\right)+ 2C_0\overline{\nu}\Delta t \|\nabla\eta_{j}^{1}\|^{2}+C_0\overline{\nu} \Delta t \Vert\nabla\eta_{j}^{0}\Vert
^{2} .
\end{aligned}
\end{equation*}

\noindent Applying inequality (\ref{ineq:errlast1}), using the previous bounds
for the $\eta_{j}^{n}$ terms, and absorbing constants into a new constant $C$, we
have Theorem \ref{th:err}.
\end{proof}

\end{document}